\documentclass[11pt,reqno]{amsart}
\usepackage{indentfirst}
\usepackage{graphics,amsmath,amssymb,amsthm,mathrsfs,bm}
\usepackage{enumitem}
\usepackage{amsfonts}
\usepackage[leqno]{amsmath}
\usepackage{amsmath}
\usepackage{amssymb}
\usepackage{multicol}
\usepackage{stmaryrd}
\usepackage{setspace}
\usepackage{bbm}
\usepackage{cite}
\usepackage{epsfig}
\usepackage{color}
\usepackage{graphics}
\usepackage{graphicx}
\usepackage{epstopdf}
\usepackage{multicol,graphics}
\usepackage{mathtools}
\usepackage{indentfirst}
\usepackage{float}

\usepackage{subfigure}
\usepackage{booktabs}

\newtheorem{theorem}{Theorem}[section]
\newtheorem{lemma}[theorem]{Lemma}
\newtheorem{corollary}[theorem]{Corollary}

\newtheorem{definition}[theorem]{Definition}
\newtheorem{remark}[theorem]{Remark}
\newtheorem{proposition}[theorem]{Proposition}
\numberwithin{equation}{section}

\allowdisplaybreaks[4]

\makeatletter
 \def\@textbottom{\vskip \z@ \@plus 30pt}
 \let\@texttop\relax
\makeatother

\usepackage{multirow}
\usepackage{arydshln}
\usepackage{hyperref}
\hypersetup{hidelinks,
colorlinks=true,linkcolor=blue,citecolor=red,
pdfstartview=Fit,
pdfborder={1 1 1},
bookmarksnumbered=true,
breaklinks=true}

\begin{document}

\title[Nonlocal dispersal predator-prey systems in shifting habitats]{\textbf{Propagation dynamics for nonlocal dispersal predator-prey systems in shifting habitats: A Hamilton-Jacobi approach}}

\author[Tao, Li, Ruan and Xu]{Wen Tao$^{1}$, Wan-Tong Li$^{1,*}$, Shigui Ruan$^{2}$ and Wen-Bing Xu$^{3}$}
\thanks{\hspace{-.1cm}
$^1$School of Mathematics and Statistics, Lanzhou University, Lanzhou, Gansu 730000, P. R. China.\\
$^2$Department of Mathematics, University of Miami, Coral Gables, FL 33146, USA.\\
$^3$School of Mathematical Sciences, Capital Normal University, Beijing 100048, P. R. China.
\\
$^*${\sf Corresponding author} (wtli@lzu.edu.cn)}

\date{\today}

\begin{abstract}
This paper is concerned with the spreading speeds of nonlocal dispersal predator-prey systems in shifting habitats under general initial conditions. By employing geometric optics techniques and theory of viscosity solutions, we reformulate the problem into the study of Hamilton-Jacobi equations. Through a detailed analysis of the structure of viscosity solutions, we provide a complete classification of explicit formulas for the spreading speed of the prey population, especially in cases where it invades the habitat more rapidly than predators, yielding two fundamentally distinct ``nonlocal determinacy'' results derived by different mechanisms. We also obtain an upper bound for spreading speed of the predators, incorporating the decay rate of the initial data and the speed of shifting habitats. These findings demonstrate that there are complex connections among spreading speeds, habitat shifting speed and initial conditions, and emphasize the significance of nonlocal dispersal in determining the propagation dynamics of predator-prey systems.



\end{abstract}
\maketitle

\section{Introduction}\label{sec-introduction}

The survival of populations in nature is always affected by climate change, which can  be caused by factors such as global warming, industrialization, and overexploitation.
Over the past decade,  the complex interplay between environmental heterogeneity driven by  climate change and population dispersal patterns has emerged as a critical focus for mathematical modeling in spatial ecological studies.
Potapov and Lewis \cite{potapov2004} and  Berestycki et al.\cite{berestycki2009,berestycki2014}   proposed the following reaction-diffusion model
\begin{align}\label{eq:rd-shift}
	u_{t}(x,t)= u_{xx}(x,t)+ u(x,t) g(x-c_{e}t, u(x,t)), \quad (x,t)\in \mathbb{R}\times\mathbb{R}^{+},
\end{align}
where  $u(x,t)$ represents the population density at position $x$ and time $t$, and the function
$g$ models the shifting habitats, depending on a moving frame of the form $x-c_{e}t$ with $c_{e}\in \mathbb{R}$ being the speed of shifting habitats. A typical example of $g$ is
\begin{align}\label{eq:growth}
	g(x-c_{e}t,u(x,t))=\alpha(x-c_{e}t)-u(x,t),
\end{align}
where $\alpha(\cdot)$ is the carrying capacity of the environment as a function of $x-c_e t$. Throughout the paper, when $\alpha(\cdot)$ is referenced, we always assume that $\alpha\in C(\mathbb{R};\mathbb{R})$ and the limits $\alpha_{-}:=\alpha(-\infty)$ and $\alpha_{+}:=\alpha(+\infty)$ exist.
In shifting habitats models,  the concept of spreading speed is usually applied to predict whether a population will survive or become extinct. For a nonnegative density function $u$, we say that it admits  {\it rightward and leftward  spreading speeds} (see \cite{aronson1975,wein1982}) if
\begin{equation*}
	\begin{cases}
 	\underset{t\rightarrow\infty}\lim \underset{x\leqslant (c_{u}^{l}-\eta)t,\,x\geqslant (c_{u}^{r}+\eta)t}\sup u(x,t)=0, &\hspace{1em} \text{for~any~}\eta>0, \\
	\underset{t\rightarrow\infty}\liminf \underset{ (c_{u}^{l}+\eta)t\leqslant x\leqslant (c_{u}^{r}-\eta)t}\inf u(x,t)>0, & \hspace{1em} \text{for~any~}\eta\in(0,(c_{u}^{r}-c_{u}^{l})/2).
\end{cases}
\end{equation*}

For the reaction-diffusion model \eqref{eq:rd-shift}-\eqref{eq:growth}  with compactly supported initial conditions, Li et al.\cite{li2014} and Hu et al.\cite{hu2019} employed the classical super-subsolution method to study the spreading speed under the condition  $\alpha_{-}>0\geqslant \alpha_{+}$. Subsequently, Yi and Zhao \cite{yi2020b} adopted the theory of monotone dynamical systems to analyze propagation dynamics in monostable  equations without spatial translational invariance, generalizing the results of \cite{li2014}.
When $\alpha_{-}>\alpha_{+}>0$, Hu et al.\cite{hu2020}  studied the spreading speed as well as the threshold behavior of solutions via the super-subsolution method. Recently, employing the theory of viscosity solutions for Hamilton-Jacobi equations, Lam and Yu \cite{lam2022a} derived explicit formulas for the spreading speed, specifically taking into account the influence of exponential decay in the initial data. We also refer to Lam et al.\cite{lam2025} and Giletti et al. \cite{giletti2024} for the case where the resource distribution has an interior maximum patch.
For reaction-diffusion Lotka-Volterra competition systems in shifting habitats, related results can be found in \cite{yuan2021,wu2022,dong2021b}, while predator-prey systems are discussed in \cite{choi2021,guo2023,lam2024}.

Motivated by the interplay between long-range population dispersal and climate change, Li et al.\cite{li2018} introduced the following nonlocal dispersal model
\begin{equation}\label{non-shifting}
	u_t(x,t)=(J*u-u)(x,t)+u(x,t)[\alpha(x-c_{e} t)-u(x,t)], \quad (x,t)\in \mathbb{R}\times\mathbb{R}^{+},
	\end{equation}
	where
	\begin{align}\label{eq:non-oper}
		(J\ast u-u)(x,t)=\int_{\mathbb{R}}J(x-y)u(y,t)\,dy-u(x,t)
	\end{align}
	represents a nonlocal diffusion operator, which captures movement between both adjacent and nonadjacent spatial locations  (see \cite{murray2003}).
Building upon the theory of semigroup for nonlocal diffusion problems and the super-subsolution method, Li et al. \cite{li2018} and Qiao et al.\cite{qiao2022} studied the spreading speeds of \eqref{non-shifting} with compactly supported initial data.
In our recent work \cite{tao-shift}, under the assumptions that $\min\left\{\alpha_{-},\alpha_{+}\right\}>0$ and $\alpha(\cdot)$ is weakly monotonic, we employed the theory of viscosity solutions for Hamilton-Jacobi equations to establish a complete classification of the linear spreading speeds in  \eqref{non-shifting}. Furthermore, the analysis reveals that when the initial decay rate falls within certain moderate ranges, the classical super-subsolution method appears to be inapplicable to \eqref{non-shifting} due to structural incompatibilities arising from the local concavity properties of viscosity solutions. This phenomena precisely illustrates the superiority of the viscosity solution framework for Hamilton-Jacobi equations over classical methods in addressing the problem \eqref{non-shifting}. Here, we also refer to Wu et al.\cite{wu2019} and Wang et al.\cite{wang2024} for the analysis of spreading speeds in nonlocal dispersal Lotka-Volterra competition systems within shifting habitats, and to Qiao et al.\cite{qiao2024} and Wang and Qiao\cite{wang2025b} for the time-periodic case. Additionally, Wang et al. \cite{wang2022} provide a relatively comprehensive review on shifting habitats.

This paper is devoted to investigating the spreading speed for the following nonlocal dispersal predator-prey system with shifting habitats
\begin{equation}\label{eq:pp-shift}
\left\{\begin{aligned}
&u_{t}=d_{1}(J_{1}\ast u-u)+r_{1}u\left[\alpha(x-c_{e}t)-u-av\right], &&(x,t)\in\mathbb{R}\times\mathbb{R}^{+}, \\
&v_{t}=d_{2}(J_{2}\ast v-v)+r_{2}v(-1+bu-v),&& (x,t)\in\mathbb{R}\times\mathbb{R}^{+},\\
&u(x,0)=u_{0}(x),\,v(x,0)=v_{0}(x),&& x\in\mathbb{R},
\end{aligned}\right.
\end{equation}
	where $u(x,t)$ and $v(x,t)$ denote the densities of the prey and predators at location $x$ and time $t$, respectively. The constant $d_{i}>0$ ($i=1,2$) are the dispersal rates,  $r_{i}>0$ ($i=1,2$) are the intrinsic growth rates, and $a,b>0$ represent  the predation rate and conversion rate, respectively. We assume  that the  spatial convolution kernels $J_{i}$ ($i=1,2$) satisfy
\begin{itemize}
	\item[\bf (J)] $J_{i} \in C\left(\mathbb{R}\right)$ are symmetric, nonnegative and compactly supported, and $\int_{\mathbb{R}} J_{i}(y) d y=1$.
	\end{itemize}
For the initial data, we always assume that  $(u_{0},v_{0})$ is nontrivial and belongs to $ \mathcal{BC}_{\alpha_{-}}\times\mathcal{BC}_{\mathcal{V}_{-}}$,	where  $\mathcal{V}_{-}:=b\alpha_{-}-1$ and
\begin{align*}
	\mathcal{B C}_{r}:=\{\phi \in C(\mathbb{R})| \phi \text { is uniformly continuous on } \mathbb{R},\,\,0 \leqslant  \phi \leqslant  r\}.
\end{align*}

It is well known that the predator-prey system \eqref{eq:pp-shift} lacks a comparison principle, posing a fundamental analytical challenge in studying its long-time dynamical behavior.
To address this, Choi et al.\cite{choi2021} investigated forced waves and spreading speeds for both nonlocal and local dispersal versions of \eqref{eq:pp-shift} under the condition $\alpha_{-}>0> \alpha_{+}$  by applying the  comparison principle for a scalar equation. Subsequently, Zhao et al. \cite{zhao2024a} extended and improved the main results in \cite{choi2021}. Their approach relied on establishing a priori estimate for solution to  \eqref{eq:pp-shift}  by imposing large diffusion conditions (see \cite[Theorems 3.1 \& 4.1]{zhao2024a} therein), thereby enabling the use of classical compactness arguments to study the spreading speeds and persistence of the prey and predators.
Recently, Zhou and Wu \cite{zhou2024a} studied \eqref{eq:pp-shift} when the environment is consistently favorable for the prey survival, specifically where the resource function $\alpha(\cdot)$ satisfies
\begin{itemize}
\item[\bf (A)]  $\alpha\in C(\mathbb{R};\mathbb{R}^{+})$ is nonincreasing and $0<\alpha_{+}<\alpha_{-}<\infty$.
		\end{itemize}
Their analysis rests on two key assumptions: first, that predators can survive at the prey's carrying capacity $\alpha_{+}$, i.e.,
\begin{itemize}
	\item[\bf (H1)]  $\mathcal{V}_{+}:= b\alpha_{+}-1>0$;
\end{itemize}
second, that both the prey and predators have sufficiently large diffusion rates, specified as
\begin{align}\label{eq:large-diffusion}
	d_{1}>r_{1}\alpha_{-}+L_{1}, \quad d_{2}>r_{2}\mathcal{V}_{-}+L_{2},
\end{align}
with some $L_{i}=L_{i}(r_{1},r_{2},a,b,\alpha_{-})>0$, $i=1,2$.
In this setting, they extended the compactness techniques developed by Ducrot et al.\cite{ducrot2019} and established the spreading speeds for both species under the condition of compactly supported initial data, see Table \ref{table-compact} for specific scenarios.

The analysis in \cite{zhou2024a} was limited to the case of favorable habitat expansion (i.e. $c_{e}>0$), where the spreading speeds of the prey and predators coincide with either the shifting speed $c_{e}$ or those of the limiting equations. In the following cases, however, the situation differs substantially.
\begin{itemize}
		\item {\bf Exponentially decaying initial data}: As mentioned above, for the scalar equation \eqref{non-shifting}, the classical super-subsolution method seems to fail for exponentially decaying initial data with small decay rates. This fundamental limitation consequently renders the scalar comparison methods employed  in \cite{zhou2024a} inapplicable to the analysis of \eqref{eq:pp-shift} with exponentially decaying initial data.
	\item {\bf More favorable habitat shrinks} (i.e., $c_e\leqslant 0$): The spreading speed of the prey will exhibit {\it nonlocal determinacy} (which will be explained later), and is not determined solely by local information at the invasion front, which significantly complicates, or even prevents, the construction of traditional super-subsolutions.
	\item {\bf Small diffusion rates}: The {\it a priori} approach employed in \cite{zhou2024a} fails, rendering the persistence and compactness methods from \cite{ducrot2019} inapplicable for the analysis of \eqref{eq:pp-shift}.
\end{itemize}
Therefore, in these scenarios, the determination of spreading speeds for \eqref{eq:pp-shift} remains an open challenge, as it is hindered by the complex interplay between the shifting speed and the decay rate of initial data.

To establish a unified analytical framework for \eqref{eq:pp-shift}, with a focus on the effects of the shifting speed and the decay rates of initial conditions while addressing the challenges mentioned above, we employ a methodology based on the theory of viscosity solutions for Hamilton-Jacobi equations to derive explicit quantitative formulas of spreading speeds.
A key feature of this technique is that it converts the original nonlocal integro-differential system into a local first-order fully nonlinear partial differential equation.
This geometric optics approach was originally established by Freidlin \cite{freidlin1985} using the probabilistic large deviation framework, whereas Evans and Souganidis \cite{evans1989} developed a theoretical foundation for analyzing reaction-diffusion equations via Hamilton-Jacobi equations from a PDE perspective.
Recently, this approach has been significantly extended to various ecological systems. Lam et al.\cite{lam2025} and Lam and Yu \cite{lam2022a} applied this framework to heterogeneous shifting habitat problems, while Liu et al.\cite{liu2020,liu2021a,liu2021}  adapted the methodology to analyze Lotka-Volterra competition systems. Additionally, Lam and Lee \cite{lam2024} used  this methodology to tackle predator-prey systems with heterogeneous conversion efficiency.  In the realm of nonlocal Lotka-Volterra parabolic equations, we refer to Lam et al.\cite{lam2023}.  Regarding applications involving spatially periodic media, Liang and Zhou\cite{liang2020a} treated the nonlocal diffusion models, and Kang and Liu\cite{kang2025} examined age-structured models using this framework. In addition, we refer to Loy and Perthame\cite{loy2024} on nonlocal kinetic equations and Henderson and Lam\cite{henderson2025} on road-field reaction-diffusion models.

We consider two types of initial data with exponential or faster decay:
	\begin{itemize}
	\item[\textbf{(I$^{\lambda}$)}] $u_{0}$ and $v_{0}$ are strictly positive on $\mathbb{R}$, and there exist   positive constants $\lambda_{i}^r$ and $\lambda_{i}^l$ with  $i=1,2$  such that
	\begin{equation}\label{initial-exp}
\begin{cases}
	u_{0}(x)\sim  e^{-\lambda_{1}^r x},\,v_{0}(x)\sim e^{-\lambda_{2}^r x} &\text{as~}x\rightarrow+\infty, \\
	u_{0}(x)\sim e^{\lambda_{1}^l x},\,v_{0}(x)\sim e^{\lambda_{2}^l x} &\text{as~}x\rightarrow-\infty.
\end{cases}
	\end{equation}
	\item[\textbf{(I$^{\infty}$)}] $u_{0}$ and $v_{0}$ satisfy
	\begin{equation*}\label{initial-compact}
	\underset{|x|\rightarrow+\infty}\limsup\,u_{0}(x) e^{\lambda |x|}=	\underset{|x|\rightarrow+\infty}\limsup\,v_{0}(x) e^{\lambda |x|}=0 \quad \text{for any }\lambda>0.
	\end{equation*}
	\end{itemize}
Any compactly supported  function  satisfies condition (I$^\infty$).
For notational consistency, we sometimes refer to ``(I$^{\infty}$)'' as ``(I$^{\lambda}$)  with $\lambda=\infty$''.
We further propose the following strong assumption on the reaction term to ensure that the predators consistently achieve successful dispersal:
\begin{itemize}
	\item[\bf (H2)]  $\alpha_{+}-a\mathcal{V}_{-}>0$, where $\mathcal{V}_{-}=b\alpha_{-}-1$.
	\end{itemize}
	This assumption also guarantees the establishment of a comparison principle for the underlying Hamilton-Jacobi equation (see Lemma \ref{lem:viscosity-comparison-t-indep}).
In particular, it follows directly that $\mathcal{V}_{-}>0$ under assumptions (A) and (H1).

Define the Hamiltonians
	\begin{align}\label{pp-hamiltonian}
		H_{1,\pm}(p)=H_{1}(p)+r_{1}\alpha_{\pm}, \quad H_{2,\pm}(p)=H_{2}(p)+r_{2}\mathcal{V}_{\pm}, \quad p\in \mathbb{R},
	\end{align}
where
\begin{align}\label{eq:hi}
	H_{i}(p)=d_{i}\left(\int_{\mathbb{R}} J_{i}(y)e^{py}\,dy-1\right), \quad i=1,2.
\end{align}
We denote the associated Lagrangians
\begin{equation}\label{eq:lagrangian}
	L_{i,\pm}(q)=\underset{p\in \mathbb{R}}\sup \{\,pq-H_{i,\pm}(p)\}, \quad  q\in \mathbb{R}, \quad i=1,2
\end{equation}
and
\begin{align}\label{eq:speed-c-pm}
	c_{i,\pm}(\mu)=\frac{H_{i,\pm}(\mu)}{\mu}, \quad \mu\in \mathbb{R}\backslash\{0\}.
\end{align}
By Lemma \ref{lem:hamilton} below, $L_{i,\pm}$ is well-defined and there exists a unique global minimum point $\mu^{*}_{i,\pm}$ of $c_{i,\pm}(\cdot)$ in $\mathbb{R}^{+}$ such that $	 c^{*}_{i,\pm}:=\underset{\mu\in\mathbb{R}^{+}}\inf\, c_{i,\pm}(\mu)$. Then we set
\begin{align}\label{eq:speed-homo}
	s_{i,\pm}^{r}= c_{i,\pm}(\lambda_{i}^r\land \mu^{*}_{1,\pm}) \quad \text{and} \quad  	s_{i,\pm}^{l}= c_{i,\pm}(\lambda_{i}^l\land \mu^{*}_{1,\pm})  \quad \text{for }\lambda_{i}^r, \lambda_{i}^l\in (0,\infty],
\end{align}
where $p\land q=\min\left\{p,q\right\}$. Similarly, we denote  $p\vee q=\max\left\{p,q\right\}$.
According to the classical theory of spreading speeds for nonlocal dispersal equations\cite{shen2010,xu2021}, $s_{1,\pm}^r$ (resp. $-s_{1,\pm}^l$) represents the rightward (resp. leftward) spreading speed of the prey in the absence of predators, under constant habitat quality $\alpha(\cdot)\equiv \alpha_{\pm}$ with initial data satisfying  (I$^{\lambda}$), $\lambda\in (0,\infty]$. Similarly, $s_{2,\pm}^r$ (resp. $-s_{2,\pm}^l$) represents the rightward (resp. leftward) spreading speed of predators when the prey density is maintained at $\alpha_{\pm}$, under initial data satisfying  (I$^{\lambda}$), $\lambda\in (0,\infty]$.

This paper focuses on the scenario where the prey spreads faster than predators, as formulated by the following assumption:
\begin{itemize}
	\item[\bf (FU)]  $s_{2,-}^{r}<s_{1,+}^{r}$ and $s_{2,-}^{l}<s_{1,+}^{l}$.
	\end{itemize}
In biological contexts, assumption (FU) indicates that prey outpace predators in spreading through the habitat in both directions. This typically occurs when predator search efficiency is low or prey reproductive rates are high. The fact that some prey diffuse faster than their predators is a common phenomenon in population dynamics. One example is that zooplankton (the prey) perform daily vertical migrations, diving to deeper and darker waters during the day to avoid fish (predators),  and ascending at night to feed. Another example is in western Canada, caribou and moose (the prey) behave faster to avoid linear features (such as roads or seismic lines) that wolves and bears (predators) use as efficient travel corridors. We exclude scenarios where $s_{2,-}^{r}\geqslant s_{1,+}^{r}$ or $s_{2,-}^{l}\geqslant s_{1,+}^{l}$, as such conditions may lead to co-dispersal between the prey and predators, meaning both species share the same spreading speed, and the analytical framework presented here is no longer applicable. This work probably constitutes a crucial advancement toward further exploring these scenarios, which we will investigate separately in future work.

\subsection{Analytical framework}
The approach adopted in this work analyzes invasion fronts macroscopically instead of tracking microscopic population densities. By respectively incorporating the shifting habitat and initial conditions into the discontinuous coefficients and boundary condition of the Hamilton-Jacobi equation, explicit formulas for the spreading speed are derived.

Applying the hyperbolic scaling from \eqref{eq:hyperbolic} in the prey equation, we obtain the corresponding singular limit problem \eqref{eq:we-viscosity} via the Hopf-Cole transformation in \eqref{eq:hopf} (a geometric optics approximation viewpoint). It can be shown that its solution $w^{\varepsilon}$ is locally uniformly bounded in $\mathbb{R}\times\mathbb{R}^{+}$.

Moreover, by the half-relaxed limits method introduced by Barles and Perthame\cite{barles1987}, we demonstrate that the lower semicontinuous envelope $w_{*}$ and the upper semi-continuous envelope $w^{*}$ of $w^{\varepsilon}$ as defined in \eqref{eq:half-relaxed} are respectively a viscosity supersolution and a viscosity subsolution of the following time-dependent Hamilton-Jacobi equation:
\begin{align*}
	\min\left\{\partial_{t}w+H_{1}(\partial_{x}w)+R(x/t),w\right\}=0, \quad (x,t)\in \mathbb{R}\times\mathbb{R}^{+},
\end{align*}
where $H_{1}(p)$ is given by \eqref{eq:hi} and $R$ is defined in Propositions \ref{prop:lower-hj-control} and \ref{prop:upper-hj-control}.
In particular, in the context of shifting habitats, the following $1$-homogeneities of $w_{*}$ and $w^{*}$ (as introduced in   \cite{liu2021,lam2022a})
\begin{align}\label{eq:1-homogeneity}
	\rho_{*}(x/t)=w_{*}(x/t,1), \quad 	\rho^{*}(x/t)=w^{*}(x/t,1)
\end{align}
are respectively a viscosity supersolution and a viscosity subsolution of the following time-independent Hamilton-Jacobi equation:
\begin{align}\label{eq:hj-t-indep-sec1}
	\min\left\{\rho(s)-s\rho'(s)+H_{1}(\rho'(s))+R(s),\rho(s)\right\}=0, \quad s\in \mathbb{R}.
\end{align}
Additionally, both $\rho_{*}$ and $\rho^{*}$ satisfy boundary conditions determined by the persistence of the prey and the decay rate of the initial data in \eqref{eq:pp-shift}, see Lemma \ref{lem:low-hj-control} and  Corollary \ref{coro:up-hj-control}.

It is important to note that equation \eqref{eq:hj-t-indep-sec1} satisfies a comparison principle on both right and left halves of the real axis, as established in Lemma \ref{lem:viscosity-comparison-t-indep}. Based on the parameter regions classified in Lemmas \ref{lem:right-division}-\ref{lem:left-division} (see Figure \ref{figure:speed-exp}), we construct a suitable viscosity solution $\rho\in C(\mathbb{R};[0,\infty))$ for \eqref{eq:hj-t-indep-sec1} and establish that
\begin{align}\label{eq:r-equiv}
	\rho\leqslant \rho_{*}\leqslant \rho^{*}\leqslant \rho \,(\Rightarrow \rho_{*}= \rho^{*}), \quad \text{in }\mathbb{R}.
\end{align}
This inequality plays a critical role  in  deriving explicit formulas for the spreading speeds of the prey. Detailed arguments are presented in Section \ref{sec:formulas-prey}.

For predators, we analyze the $u$-equation and utilize the stability properties of viscosity solutions to derive the following Hamilton-Jacobi equation:
\begin{align*}
	\min\left\{\rho-s\rho'(s)+H_{2}(\rho'(s))+R_{0}(s),\rho\right\}=0, \quad s\in \mathbb{R},
\end{align*}
with the boundary conditions \eqref{eq:hj-tindep-low-boundary-I},
where $H_{2}(p)$ and $R_{0}$ are respectively given by \eqref{eq:hi} and \eqref{eq:R0-usc}, see Subsection \ref{subsec:v-low-hj} for details.
By constructing viscosity subsolutions, we obtain an upper bound for the spreading speed of predators, as presented in Section \ref{sec:formula-max-predator}. However, the precise spreading speed of predators remains undetermined now, and is deferred to future work.

\begin{remark}
	\rm{
By \eqref{eq:1-homogeneity} and the definition of the half-relaxed limit in \eqref{eq:half-relaxed}, we have $\rho_{*}\leqslant \rho^{*}$ on $\mathbb{R}$. However, since $\rho_{*}$ and $\rho^{*}$ satisfy different Hamilton-Jacobi equations \eqref{eq:hj-tindep-up} and \eqref{eq:hj-tindep-low}, respectively,  we cannot directly apply Lemma \ref{lem:viscosity-comparison-t-indep} to deduce that $\rho_{*}\geqslant  \rho^{*}$.
Instead, we invoke a sufficient (though possibly stronger than necessary) condition (FU) to establish \eqref{eq:r-equiv}. This, in turn, implies that as $\varepsilon\rightarrow0$,
		\begin{align}\label{eq:we-converge}
			w^{\varepsilon}\rightarrow t\rho_{*}=t\rho^{*}=t\rho \quad \text{locally uniformly in } \mathbb{R}\times\mathbb{R}^{+}.
		\end{align}
See details in Corollary \ref{coro:unique-vis}. In particular, combining with \eqref{eq:we-converge}, we obtain the large deviation estimate for $u$:
		\begin{align}\label{eq:large-devi}
			 u(ct, t)\sim e^{-t\rho(c)+o(t)} \quad \text{as}~ t\rightarrow\infty, \quad \text{for any } c\in\{s\in \mathbb{R}|\rho(s)>0\}.
		\end{align}
	}
\end{remark}
\begin{table}[htpb]
		\begin{center}
			\caption{Spreading speed of \eqref{eq:pp-shift} with compactly supported initial data.}
			\renewcommand{\arraystretch}{1.2}
\begin{tabular}{c|cccc}
 \hline
\multirow{2}{*}{Various scenarios}  & \multicolumn{2}{c|}{Spreading speed of the prey}   & \multicolumn{2}{c}{Spreading speed of predators}   \\ \cline{2-5}
  & \multicolumn{1}{c:}{$c_{e}>0$} & \multicolumn{1}{c|}{$c_{e}\leqslant 0$} & \multicolumn{1}{c:}{$c_{e}>0$}   & $c_{e}\leqslant 0$    \\
	\hline
\multirow{3}{*}{$c^{*}_{2,+}<c^{*}_{2,-}<c^{*}_{1,+}<c^{*}_{1,-}$} & \multicolumn{1}{c:}{\multirow{4}{*}{See \cite{zhou2024a}$^{*,a}$}} & \multicolumn{1}{c|}{\multirow{2}{*}{This work provides} } & \multicolumn{1}{c:}{\multirow{4}{*}{See \cite{zhou2024a}$^{*,b}$}} & \multirow{2}{*}{This work provides}\\
  & \multicolumn{1}{c:}{}   &   \multicolumn{1}{c|}{\multirow{2}{*}{a classification}} & \multicolumn{1}{c:}{}        &   \multirow{2}{*}{upper bounds}  \\
  & \multicolumn{1}{c:}{}   & \multicolumn{1}{c|}{}    & \multicolumn{1}{c:}{}  & \multirow{1}{*}{}   \\
	\cline{1-1}  \cline{3-3} \cline{5-5}
$c^{*}_{2,+}<c^{*}_{1,+}<c^{*}_{2,-}<c^{*}_{1,-}$ & \multicolumn{1}{c:}{} & \multicolumn{1}{c|}{\multirow{1}{*}{Open} } & \multicolumn{1}{c:}{}    &   Open \\
 \hline
$c^{*}_{2,+}<c^{*}_{1,+}<c^{*}_{1,-}<c^{*}_{2,-}$ & \multicolumn{4}{c}{\multirow{2}{*}{Open}} \\
\cline{1-1}
$c^{*}_{1,+}<c^{*}_{2,+}<c^{*}_{2,-}<c^{*}_{1,-}$&  \multicolumn{4}{c}{}     \\
 \hline
$c^{*}_{2,\pm}>c^{*}_{1,\pm}$& \multicolumn{1}{c:}{See \cite{zhou2024a}} & \multicolumn{1}{c|}{Open}  & \multicolumn{1}{c:}{See \cite{zhou2024a}} & Open \\
 \hline
\end{tabular}
\end{center}
\begin{flushleft}
$^{*,a}$ All critical cases are addressed in this paper, see Theorem \ref{thm:u-speed-comp}. \\
$^{*,b}$ Upper bounds are provided in this paper for all  critical cases, see Theorems \ref{thm:v-maxspeed-right}-\ref{thm:v-maxspeed-left}.
\end{flushleft}
	\label{table-compact}
\end{table}

\begin{table}[htpb]
	\begin{center}
			\caption{Spreading speed of \eqref{eq:pp-shift} with exponentially decaying initial data.}
			\renewcommand{\arraystretch}{1.3}
\begin{tabular}{c|cc}
\hline
 Various scenarios & \multicolumn{1}{c|}{Spreading speed of the prey} & Spreading speed of predators \\ \hline
\multirow{2}{*}{$s_{2,-}^{r}<s_{1,+}^{r}$,\,$s_{2,-}^{l}<s_{1,+}^{l}$} & \multicolumn{1}{c|}{\multirow{1}{*}{This work provides}}    & \multirow{1}{*}{This work provides}   \\ [-5pt]
 & \multicolumn{1}{c|}{a classification}   &   \multicolumn{1}{c}{\multirow{1}{*}{upper bounds} } \\ \hline
 Otherwise& \multicolumn{2}{c}{Open}                                                  \\ \hline
\end{tabular}
\label{table-exp}
\end{center}
\end{table}

\subsection{Advances and new phenomena}
Our work appears to be the first to analyze the long-time dynamics of a nonlocal dispersal predator-prey system using the theory of viscosity solutions.
This work makes significant advances in the analysis of spreading speed for \eqref{eq:pp-shift}, providing a complete description for several cases that had eluded a complete characterization in past studies. We also reveal some previously uncharacterized propagation phenomena in system \eqref{eq:pp-shift}. These phenomena are often difficult to address within the framework of traditional super-subsolution methods and compactness techniques, particularly in nonlocal dispersal systems. Our work therefore establishes a new exploration in this emerging research direction.

\begin{itemize}
	\item The spreading speed of the prey is subject to \textbf{nonlocal determinacy}. Our results in Theorems \ref{thm:u-speed-exp-right}-\ref{thm:u-speed-comp} indicate that the spreading speed of the prey is determined not merely by the decay of initial data,  but by the complex coupling of $c_{e}$ and two limiting states (i.e. the case of  $\alpha(\cdot)\equiv \alpha_{-}$ or $\alpha(\cdot)\equiv \alpha_{+}$).
 In particular, within the regions $V_{1,r}^d$, $V_{1,l}^b$ and $V_{1,l}^c$ (see the colored areas in Figure \ref{figure:speed-exp}), the spreading speed of the prey does not align solely with the shifting speed $c_{e}$ or the spreading speed of the limiting equation. Instead,
	it emerges as a global phenomenon that cannot be inferred from local conditions at the invasion front alone; we therefore term this \textbf{nonlocal determinacy} and call the resulting speed the \textbf{nonlocally determined speed}. This nonlocal behavior is reflected in a dynamic adjustment of the decay rate of $u$ near the shifting environmental front, an adjustment driven by the globally heterogeneous distribution of the resource function $\alpha(\cdot)$.
Notably, during leftward propagation,  such nonlocal determinacy can occur for any initial decay rate. However, in rightward propagation, it arises exclusively when the initial decay rate is sufficiently small (i.e. $\lambda_{1}^r<\mu^{*}_{1,-}$).

Moreover, for a given $\lambda_{1}^r<\mu^{*}_{1,-}$, the prey's rightward spreading speed exhibits a transition among several distinct regimes as $c_{e}$ increases: critical or non-critical wave speed of the limiting equation, the habitat shifting speed $c_{e}$, and the nonlocally determined speed.
A similar transition has been reported for the scalar equation \eqref{eq:rd-shift} in \cite{lam2025} and for \eqref{non-shifting} in our previous work \cite{tao-shift} .
	\item Theorems \ref{thm:u-speed-exp-right}-\ref{thm:u-speed-comp} establish a complete characterization of the linear spreading speed for the prey. In particular, under the initial conditions satisfying {\rm (I$^{\infty}$)}, Theorem \ref{thm:u-speed-comp} establishes the spreading speed for both $c_e>0$ and $c_{e}\leqslant 0$, the latter of which was not addressed in  \cite{zhou2024a}. A significant finding is that when $c_{e}\leqslant 0$, the prey  exhibits a nonlocally determined spreading speed, a novel phenomenon observed for \eqref{eq:pp-shift}.  Furthermore, our analytical framework achieves these results without requiring the large diffusion rate conditions \eqref{eq:large-diffusion} imposed in \cite{zhou2024a}.

	\item Based on the large deviation estimate \eqref{eq:large-devi} and the construction of viscosity solutions, our results indicate that when $c_{e}>s_{1,+}^r$ (or $c_{e}<-s_{1,-}^l$), the rightward (or leftward) spreading speed is  $c_{u}^r$ (or $c_{u}^l$). However, this does not imply that the solution of the Cauchy problem with front-like initial data converges to a traveling wave with speed $c_{u}^r$ (or $c_{u}^l$), due to differences in their decay behavior as $x\to+\infty$ (or $x\to-\infty$).
\end{itemize}
These novel observations, along with the remaining open questions in Tables \ref{table-compact} and \ref{table-exp}, will inspire deeper research into model \eqref{eq:pp-shift}.

\subsection{Main theorems} Now we introduce the necessary notations  and present the main theorems. Since $H_{i}$ given by \eqref{eq:hi} is convex and smooth on $\mathbb{R}$, so is its Lagrangian $L_{i}$ in \eqref{eq:lagrangian}. Therefore, the first order derivative of $H_{i}$ (or $L_{i}$) denoted by $H_{i}'$ (or $L_{i}'$) exists and is strictly increasing on $\mathbb R$ (see Lemma \ref{lem:hamilton}).
When $(c_{e},\lambda_{i}^{r})\in \hat{\mathcal{R}}_{i}\cap\{(c_{e},\lambda_{i}^{r})|c_{e}\lambda_{i}^{r}-H_{i,+}(\lambda_{i}^{r})< L_{i,-}(c_{e})\}$, where $\hat{\mathcal{R}}_{i}$ is given by
\begin{align}\label{eq:region-R}
	 \hat{\mathcal{R}}_{i}:=\{(c_{e},\lambda_{i}^{r})|c_{e}>c_{i,+}(\lambda_{i}^{r}) \text{~for~}\lambda_{i}^{r}\leqslant \mu_{0,i}\text{~or~} c_{e}>c^{*}_{i,+} \text{~for~}\lambda_{i}^{r}\in(\mu_{0,i},\infty)\},
\end{align}
and $\mu_{0,i}$ is the smallest root of $c^{*}_{i,-}=c_{i,+}(\mu)$ in $\mathbb{R}^{+}$,
let $p^{*}_{i}=p^{*}_{i}(c_{e},\lambda_{i}^{l})$ be the smallest root of
\begin{align}\label{eq:root-left-intro}
c_{e} p-H_{i,-}(p)=c_{e}\lambda_{i}^{r}-H_{i,+}(\lambda_{i}^{r}).
\end{align}
For $c_e\in
\overline{\mathcal{E}}_{i}:=\{c_{e}|c_{e}\leqslant -c^{*}_{i,-}=-H'_{i}(\mu^{*}_{i,-})\}$, let  $\bar{p}_{i}=\bar{p}_{i}(c_{e})$ be the smallest root of
\begin{equation}\label{eq:p-bar-root}
	-c_{e}p-H_{i,+}(p)=L_{i,-}(c_{e}).
\end{equation}
Then there exists a unique constant  $\bar{c}_{i}>c^{*}_{i,-}$ such that $\bar{p}_{i}(\bar{c}_{i})=\mu^{*}_{i,+}$(see Lemma \ref{lem:left-para-p}). Besides, for $(c_{e},\lambda^{l}_{i})\in\underline{\mathcal{E}}_{i}:=\{(c_{e},\lambda^{l}_{i})\,|\,c_{e}\leqslant -c_{i,-}(\lambda^{l}_{i}) \text{~when~}\lambda^{l}_{i}\leqslant \mu^{*}_{i,-},\text{~or~}c_{e}\leqslant -H_{i}'(\lambda^{l}_{i}) \text{~when~}\lambda^{l}_{i}>\mu^{*}_{i,-}\}$, let $\underline{p}_{i}=\underline{p}_{i}(c_{e},\lambda^{l}_{i})$ be the smallest root of
\begin{equation}\label{eq:p-under-root}
	-c_{e}p-H_{i,+}(p)=-c_{e}\lambda^{l}_{i}-H_{i,-}(\lambda^{l}_{i}).
\end{equation}

\begin{figure}[htpb]
	\centering
	\subfigure[Rightward division.]{\includegraphics[width=0.43\textwidth]{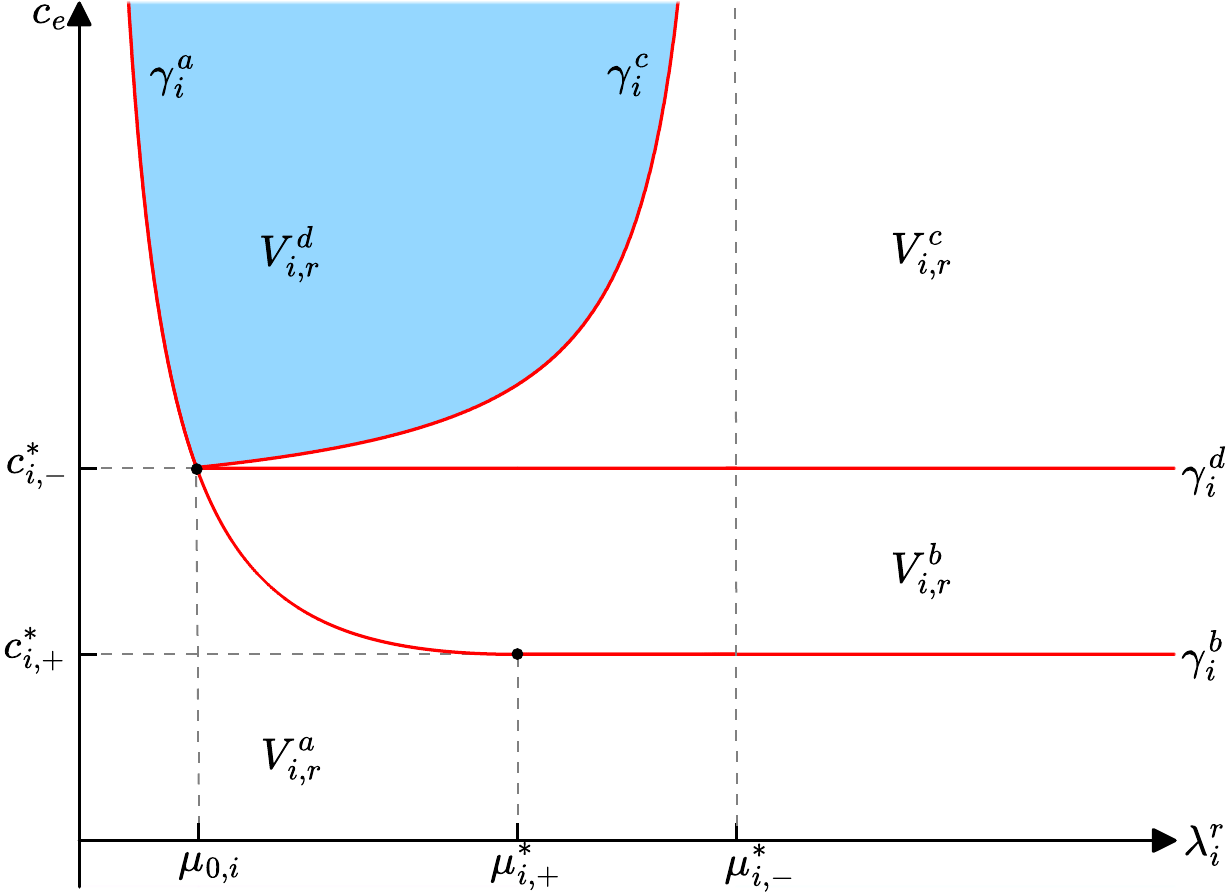}}\hspace{1.5em}
	\subfigure[Leftward division.]{\includegraphics[width=0.43\textwidth]{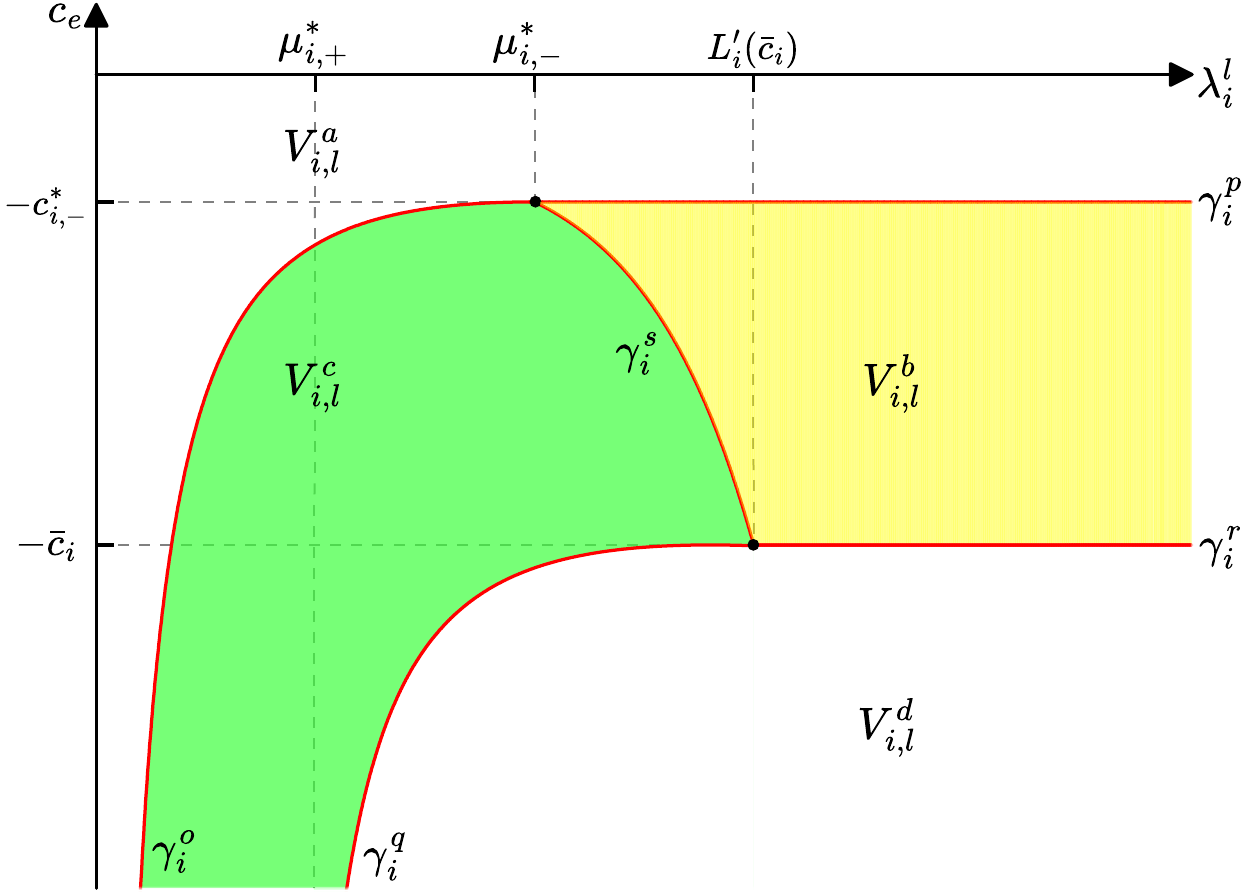}}
	\caption{Division of $E_{i}^{r}:=\{(\lambda_{i}^{r}, c_e)\}$ and $E_{i}^{l}:=\{(\lambda_{i}^{l}, c_e)\}$ by Lemmas \ref{lem:right-division}-\ref{lem:left-division}. }
	\label{figure:speed-exp}
	\end{figure}

The following three theorems provide explicit formulas for the spreading speed of the prey.

\begin{theorem}\label{thm:u-speed-exp-right}
	Assume that  {\rm (J)}, {\rm (A)}, {\rm (H1)-(H2)}, {\rm (FU)} and {\rm (I$^{\lambda}$)} hold. Let $\gamma_{1}^{a},\gamma_{1}^{b},\cdots, \gamma_{1}^{d}$ and  $V_{1,r}^{a}$, $V_{1,r}^{b}$,$\cdots$,$V_{1,r}^{d}$ be given by Lemma \ref{lem:right-division} (see also Figure \ref{figure:speed-exp}-(a)).
Then the rightward spreading speed of the prey satisfies the  following alternatives.
\begin{equation*}
		c_{u}^{r}=
	\begin{cases}
	s_{1,+}^{r}	 &\text{if } (\lambda_{1}^r, c_e)\in V_{1,r}^{a}\cup \gamma_{1}^{a}\cup \gamma_{1}^{b},\\
 c_{e} &\text{if } (\lambda_{1}^r, c_e)\in V_{1,r}^{b}\cup \gamma_{1}^{d},\\
	c_{1,-}^{*} &\text{if }  (\lambda_{1}^r, c_e)\in V_{1,r}^{c}\cup \gamma_{1}^{c},\\
		c_{1,-}(p^{*}_{1}) &\text{if }(\lambda_{1}^r, c_e)\in V_{1,r}^{d}.
	\end{cases}
	\end{equation*}
\end{theorem}

\begin{theorem}\label{thm:u-speed-exp-left}
	Assume that  {\rm (J)}, {\rm (A)}, {\rm (H1)-(H2)}, {\rm (FU)} and {\rm (I$^{\lambda}$)} hold. Let $\gamma_{1}^{o},\gamma_{1}^{p},\cdots, \gamma_{1}^{s}$ and  $V_{1,l}^{a}$, $V_{1,l}^{b}$,$\cdots$,$V_{1,l}^{d}$ be given by Lemma \ref{lem:left-division} (see also Figure \ref{figure:speed-exp}-(b)).
	Then the leftward spreading speed of the prey satisfies the  following alternatives.
		\begin{equation*}
		c_{u}^{l}=
	\begin{cases}
	-s_{1,-}^{l}	 &\text{if } (\lambda_{1}^l,c_{e})\in V_{1,l}^{a}\cup \gamma_{1}^{o}\cup \gamma_{1}^{p},\\
 -c_{1,+}(\underline{p}_{1}) &\text{if } (\lambda_{1}^l,c_{e})\in V_{1,l}^{c},\\
	-c_{1,-}(\bar p_{1}) &\text{if } (\lambda_{1}^l,c_{e})\in V_{1,l}^{b}\cup \gamma_{1}^{s},\\
	-c_{1,+}^{*} &\text{if } (\lambda_{1}^l,c_{e})\in V_{1,l}^{d}\cup \gamma_{1}^{q}\cup \gamma_{1}^{r}.
	\end{cases}
	\end{equation*}
\end{theorem}

\begin{theorem}\label{thm:u-speed-comp}
	Assume that  {\rm (J)}, {\rm (A)}, {\rm (H1)-(H2)}, {\rm (FU)} and {\rm (I$^{\infty}$)} hold. Then the rightward and leftward spreading speeds of the prey  satisfy the  following alternatives.
	\begin{eqnarray*}
&c_{u}^{r}=\left\{\begin{aligned}
&c^{*}_{1,-}&&\text{if }c_{e}\geqslant c^{*}_{1,-},\\
&c_{e}&&\text{if }c_{e}\in(c^{*}_{1,+},c^{*}_{1,-}),\\
&c^{*}_{1,+}&&\text{if }c_{e}\leqslant c^{*}_{1,+},
\end{aligned}
\right. \quad
&c_{u}^l=\left\{\begin{aligned}
&-c^{*}_{1,-}&&\text{if }c_{e}\geqslant -c^{*}_{1,-},\\
&-c_{1,-}(\bar{p}_{1})&&\text{if }c_{e}\in(-\bar{c}_{1},-c^{*}_{1,-}),\\
&-c^{*}_{1,+}&&\text{if }c_{e}\leqslant -\bar{c}_{1}.
\end{aligned}
\right.
\end{eqnarray*}
\end{theorem}

\begin{figure}[htpb]
	\centering
	\includegraphics[height=2.2in,trim=2 40 0 0, clip]{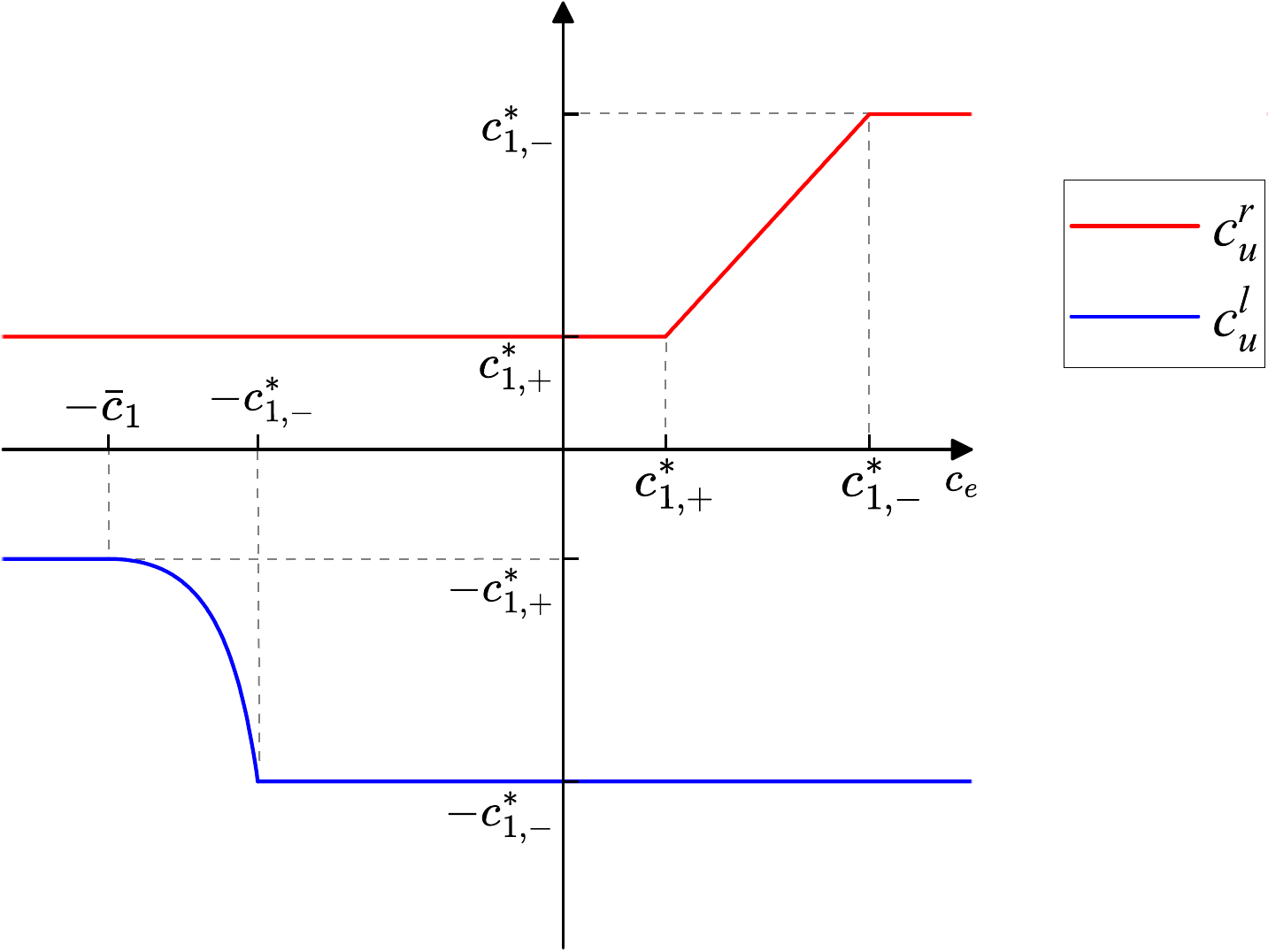}
	\caption{Spreading speeds of the prey with respect to $c_{e}$ when ($\text{I}^\infty$) holds.
	}
	\label{figure:speed-com}
\end{figure}

The results in Theorems \ref{thm:u-speed-exp-right}, \ref{thm:u-speed-exp-left} and \ref{thm:u-speed-comp} provide a complete characterization of  the dependence of the rightward and  leftward spreading speeds on the shifting sped $c_e$ and  the decay rates of the initial data under conditions {\rm (I$^{\lambda}$)} or {\rm (I$^{\infty}$)}.
Moreover, these spreading speeds are continuous with respect to $c_{e}\in \mathbb{R}$ and $\lambda_{1}^r,\lambda_{1}^l\in (0,\infty]$. An illustrative diagram of this dependence under condition  ($\text{I}^\infty$) is given by Figure \ref{figure:speed-com}.

The following theorem further addresses the persistence of the prey.

\begin{theorem}\label{thm:u-terrace}
	Assume that  {\rm (J)}, {\rm (A)}, {\rm (H1)-(H2)}, {\rm (FU)} and {\rm (I$^{\lambda}$)} with $\lambda\in(0,\infty]$ hold. Then we have the following alternatives for the propagating terrace of the prey.
\begin{enumerate}
	\item[\rm(a)] If $c_{e}\geqslant s_{1,+}^r$, then
	\begin{equation}\label{eq:u-terrace-1}
	\begin{cases}
		\underset{t\rightarrow+\infty}\lim\, \underset{ (s_{2,-}^r+\eta)t \leqslant x \leqslant (c_{u}^r-\eta)t}\sup\, |u(x,t)-\alpha_{-}|=0, & \forall \eta\in (0,(c_{u}^r-s_{2,-}^r)/2),\\
		\underset{t\rightarrow+\infty}\lim\, \underset{ (-s_{1,-}^l+\eta)t \leqslant x \leqslant (-s_{2}^l-\eta)t}\sup\, |u(x,t)-\alpha_{-}|=0, & \forall \eta\in (0,(s_{1,-}^l-s_{2,-}^l)/2).
	\end{cases}
	\end{equation}
	\item[\rm(b)] If $s_{2,-}^r<c_{e}< s_{1,+}^r$, then
	\begin{equation*}\label{eq:u-terrace-2}
	\begin{cases}
		\underset{t\rightarrow+\infty}\lim\, \underset{ (c_{e}+\eta)t \leqslant x \leqslant (s_{1,+}^r-\eta)t}\sup\, |u(x,t)-\alpha_{+}|=0, & \forall \eta\in (0,(s_{1,+}^r-c_{e})/2),\\
		\underset{t\rightarrow+\infty}\lim\, \underset{ (s_{2,-}^r+\eta)t \leqslant x \leqslant (c_{e}-\eta)t}\sup\, |u(x,t)-\alpha_{-}|=0, & \forall \eta\in (0,(c_{e}-s_{2,-}^r)/2),\\
		\underset{t\rightarrow+\infty}\lim\, \underset{ (-s_{1,-}^l+\eta)t \leqslant x \leqslant (-s_{2}^l-\eta)t}\sup\, |u(x,t)-\alpha_{-}|=0, & \forall \eta\in (0,(s_{1,-}^l-s_{2,-}^l)/2).
	\end{cases}
	\end{equation*}
	\item[\rm(c)] If $-s_{1,-}^l<c_{e}<-s_{2,-}^l$, then
	\begin{equation}\label{eq:u-terrace-3}
		\begin{cases}
			\underset{t\rightarrow+\infty}\lim\, \underset{ (s_{2,-}^r+\eta)t \leqslant x \leqslant (s_{1,+}^r-\eta)t}\sup\, |u(x,t)-\alpha_{+}|=0, & \forall \eta\in (0,(s_{1,+}^r-s_{2,-}^r)/2),\\
			\underset{t\rightarrow+\infty}\lim\, \underset{ (c_{e}+\eta)t \leqslant x \leqslant (-s_{2,-}^l-\eta)t}\sup\, |u(x,t)-\alpha_{+}|=0, & \forall \eta\in (0,(-s_{2,-}^l-c_{e})/2),\\
			\underset{t\rightarrow+\infty}\lim\, \underset{ (-s_{1,-}^l+\eta)t \leqslant x \leqslant (c_{e}-\eta)t}\sup\, |u(x,t)-\alpha_{-}|=0, & \forall \eta\in (0,(s_{1,-}^l+c_{e})/2).
		\end{cases}
		\end{equation}
	\item[\rm(d)] If $c_{e}\leqslant -s_{1,-}^l$, then
	\begin{equation*}\label{eq:u-terrace-4}
		\begin{cases}
			\underset{t\rightarrow+\infty}\lim\, \underset{ (s_{2,-}^r+\eta)t \leqslant x \leqslant (s_{1,+}^r-\eta)t}\sup\, |u(x,t)-\alpha_{+}|=0, & \forall \eta\in (0,(s_{1,+}^r-s_{2,-}^r)/2),\\
			\underset{t\rightarrow+\infty}\lim\, \underset{ (c_{u}^l+\eta)t \leqslant x \leqslant (-s_{2,-}^l-\eta)t}\sup\, |u(x,t)-\alpha_{+}|=0, & \forall \eta\in (0,(-s_{2,-}^l-c_{u}^l)/2),\\
		\end{cases}
		\end{equation*}
\end{enumerate}
where $c_{u}^r$ and $c_{u}^l$ are given by Theorems \ref{thm:u-speed-exp-right}-\ref{thm:u-speed-comp}.
\end{theorem}

For predators, the following two theorems provide respectively the upper bounds for the rightward and leftward spreading speeds (where the negative sign merely indicates the direction).

\begin{theorem}\label{thm:v-maxspeed-right}
	Assume that  {\rm (J)}, {\rm (A)}, {\rm (H1)} and {\rm (I$^{\lambda}$)} with $\lambda\in(0,\infty]$ hold. Let $\gamma_{2}^{a},\gamma_{2}^{b},\cdots, \gamma_{2}^{d}$ and  $V_{2,r}^{a}$, $V_{2,r}^{b}$,$\cdots$,$V_{2,r}^{d}$ be given by Lemma \ref{lem:right-division}.
	Then the rightward spreading speed of predators satisfies  following alternatives.
	\begin{enumerate}
		\item[\rm(a)] If $(\lambda_{2}^r, c_e)\in V_{2,r}^{a}\cup \gamma_{2}^{a}\cup \gamma_{2}^{b}$, then   $ c_{v}^{r}\leqslant s_{2,+}^r$.
		\item[\rm(b)] If $(\lambda_{2}^r, c_e)\in V_{2,r}^{b}\cup \gamma_{2}^{d}$, or $\lambda_{2}^r=\infty$ and $c^{*}_{2,+}<c_{e}\leqslant c^{*}_{2,-}$, then  $c_{v}^{r}\leqslant c_{e}$.
		\item[\rm(c)] If $(\lambda_{2}^r, c_e)\in V_{2,r}^{c}\cup \gamma_{2}^{c}$, or $\lambda_{2}^r=\infty$ and $c_{e}> c^{*}_{2,-}$, then   $ c_{v}^{r}\leqslant c_{2,-}^{*}$.
		\item[\rm(d)] If $(\lambda_{2}^r, c_e)\in V_{2,r}^{d}$,  then $ c_{v}^{r}\leqslant c_{2,-}(p^{*}_{2})$.
	\end{enumerate}
\end{theorem}

\begin{theorem}\label{thm:v-maxspeed-left}
	Assume that  {\rm (J)}, {\rm (A)},  {\rm (H1)} and {\rm (I$^{\lambda}$)} with $\lambda\in(0,\infty]$ hold. Let $\gamma_{2}^{o},\gamma_{2}^{p},\cdots, \gamma_{2}^{s}$ and  $V_{2,l}^{a}$, $V_{2,l}^{b}$,$\cdots$,$V_{2,l}^{d}$ be given by Lemma \ref{lem:left-division}.	Then the leftward spreading speed of predators satisfies  following alternatives.
	\begin{enumerate}
		\item[\rm(a)] If $(\lambda_{2}^l,c_{e})\in V_{2,l}^{a}\cup \gamma_{2}^{o}\cup \gamma_{2}^{p}$, then  $ c_{v}^{l} \geqslant- s_{2,-}^{l}$.
		\item[\rm(b)]  If $(\lambda_{2}^l,c_{e})\in V_{2,l}^{c}$,	 then   $ c_{v}^{l} \geqslant -c_{2,+}(\underline{p}_{2})$.
		\item[\rm(c)] If $(\lambda_{2}^l,c_{e})\in V_{2,l}^{b}\cup \gamma_{2}^{s}$, or $\lambda_{2}^l=\infty$ and $-\bar{c}_{2}<c_{e}<-c^{*}_{2,-}$, then   $ c_{v}^{l} \geqslant -c_{2,-}(\bar p_{2})$.
		\item[\rm(d)] If $(\lambda_{2}^l,c_{e})\in V_{2,l}^{d}\cup \gamma_{2}^{q}\cup \gamma_{2}^{r}$, or $\lambda_{2}^l=\infty$ and $c_{e}<-\bar{c}_{2}$, then   $ c_{v}^{l} \geqslant -c_{2,+}^{*}$.
	\end{enumerate}
\end{theorem}

In estimating the upper bound of the predators' spreading speed, our analysis considers only the habitat shift speed $c_{e}$ and the predators' initial decay condition, excluding the prey's spreading speed.
Notably, since predators cannot survive without the prey, their spreading speed cannot surpass that of the prey. As the prey spreads, it effectively creates a virtual shifting habitat enabling predator survival.
We conjecture that, under assumption (FU), the upper bounds established in Theorems \ref{thm:v-maxspeed-right}-\ref{thm:v-maxspeed-left} precisely characterize the predators' actual spreading speeds in some scenarios. However, as this requires overcoming challenges in analyzing the coupling effect of the prey near the propagation front, which falls beyond the direct applicability of the current Hamilton-Jacobi framework, rigorously proving this conjecture remains an open challenge.

\textbf{Outline of the paper.} Section \ref{sec:preli} contains preliminary materials. Section \ref{sec:hamilton-limit} is devoted to the study of  the limiting Hamilton-Jacobi equation for the prey. In Section \ref{sec:formulas-prey}, by constructing viscosity solutions, we derive explicit formulas for the spreading speed of the prey and prove Theorems \ref{thm:u-speed-exp-right}-\ref{thm:u-speed-comp}. Section \ref{sec:formula-max-predator} applies ideas from the preceding sections to establish upper bounds for the spreading speed of predators, proving Theorems \ref{thm:v-maxspeed-right}-\ref{thm:v-maxspeed-left}. In Section \ref{sec:terrace-prey}, we investigative the persistence of the prey and prove Theorem \ref{thm:u-terrace}.

\section{Preliminaries}\label{sec:preli}

\subsection{Division of the parameter area}

\begin{lemma}\label{lem:hamilton}
Assume that {\rm (A) and \rm (H1) hold}.	Let $H_{i,\pm}$ be given by \eqref{pp-hamiltonian} with $i=1,2$.
	Then $H_{i,\pm}$ is symmetric, superlinear and strictly convex on $\mathbb{R}$, and thus the associated Lagrangian \eqref{eq:lagrangian}	 is well-defined, symmetric, superlinear and strictly convex   on $\mathbb{R}$.  Moreover, $c_{i,\pm}$ defined by \eqref{eq:speed-c-pm}  is strictly decreasing when $\mu\in(0,\mu^{*}_{i,\pm})$ and strictly increasing when $\mu\in(\mu^{*}_{i,\pm},+\infty)$, where $\mu^{*}_{i,\pm}$ is the unique global minimum point of $c_{i,\pm}$ in $\mathbb{R}^{+}$ such that
			\begin{align*}
				 c^{*}_{i,\pm}:=c_{i,\pm}(\mu^{*}_{i,\pm})=\underset{\mu\in\mathbb{R}^{+}}\inf\, c_{i,\pm}(\mu).
			\end{align*}
In particular, it holds that
	\begin{equation}\label{eq:lag-pm}
		L_{i,\pm}(q)\leqslant 0 \quad \text{if~and~only~if}\quad|q|\leqslant c_{i,\pm}^{*},
	\end{equation}
	and the equality holds only if $|q|=c_{i,\pm}^{*}$. Besides, the Lagrangian $L_{i,\pm}$ can be expressed as
	\begin{equation*}
		L_{i,\pm}(q)=qL_{i}'(q)-H_{i,\pm}(L_{i}'(q))\quad \text{for~}q\in \mathbb{R}.
	\end{equation*}
\end{lemma}
The  proof of this lemma follows from a slight modification of the proof of \cite[Proposition 1.3 and Lemma 4.1]{tao2025}, and thus, we omit the details here.

Define two auxiliary functions
\begin{align}
	\label{k-lam}	& k_{i}(\mu)=\frac{H_{i,-}(\mu^{*}_{i,-})-H_{i,+}(\mu)}{\mu^{*}_{i,-}-\mu} \quad \text{for }\mu\in(0,\mu^{*}_{i,-}), \\
	\label{g-lam}		 &g_{i}(\mu)=\frac{H_{i,-}(\mu)-H_{i,+}(\mu^{*}_{i,+})}{\mu-\mu^{*}_{i,+}} \quad \text{for }\mu\in(\mu^{*}_{i,+},L'_{i}(\bar c_{i})],
	\end{align}
	where $\bar{c}_{i}$ is given by Lemma \ref{lem:left-para-p}. Their properties are summarized in the following lemma.
	\begin{lemma}
		\label{lem:well-division}
The following assertions hold:
	\begin{enumerate}
		\item[\rm(a)] The function $k_{i}(\cdot)$ defined by \eqref{k-lam} is strictly increasing. In particular, $k_{i}(\mu_{0,i})=c^{*}_{i,-}$ and $ \lim_{\mu\rightarrow\mu^{*}_{i,-}}\,  k_{i}(\mu)=+\infty$.
		\item[\rm(b)] The function $g_{i}(\cdot)$ defined by \eqref{g-lam} is strictly decreasing. In particular, $g_{i}(L'_{i}(\bar{c}_{i}))=\bar{c}_{i}$. For $\mu\in(\mu^{*}_{1,+},\mu^{*}_{1,-})$, the intervals $(c_{i,-}(\mu), g_{i}(\mu))$ is nonempty, and $(H_{i}'(\mu), g_{i}(\mu))$ for $\mu\in(\mu^{*}_{i,-},L'_{i}(\bar{c}_{i}))$ is also nonempty.
	\end{enumerate}
\end{lemma}
This lemma follows from straightforward calculus; we omit the details here, but refer the reader to \cite{tao-shift} for further elaboration.
	We are now ready to present two lemmas that partition the parameters $(\lambda_{i}^{r}, c_e)$ and $(\lambda_{i}^{l}, c_e)$ into several regions.
\begin{lemma}\label{lem:right-division}
For $i=1,2$, define the curves $\gamma_{i}^{a},\gamma_{i}^{b},\cdots, \gamma_{i}^{d}$ as follows:
	\begin{eqnarray*}
		&&	\gamma_{i}^{a}: c_{e}=c_{i,+}(\lambda^{r}_{i}) \hspace{1.7em}\text{ for }\lambda^{r}_{i}\in (0,\mu^{*}_{i,+}],\\
		&&	\gamma_{i}^{b}: c_{e}\equiv c^{*}_{i,+} \hspace{3.5em}\text{ for }\lambda^{r}_{i}\in [\mu^{*}_{i,+},\infty),\\
	&&	\gamma_{i}^{c}:c_{e}=k_{i}(\lambda^{r}_{i}) \hspace{2.3em}\text{ for }\lambda^{r}_{i}\in [\mu_{0,i},\mu^{*}_{i,-}),\,\\
		&&	\gamma_{i}^{d}:c_{e}\equiv c^{*}_{i,-}\hspace{3.2em}\text{ for }\lambda^{r}_{i}\in [\mu_{0,i},\infty).
		\end{eqnarray*}
	Then the half-plane $E_{i}^r:=\{(\lambda_{i}^r, c_e)|\lambda_{i}^r\in(0,\infty),~c_e\in\mathbb R\}$  is divided into the following four disjoint regions:
\begin{eqnarray*}
	&& V_{i,r}^{a}: 	\{(\lambda_{i}^r, c_e)|\lambda_{i}^r\in (0,\infty) \text{ and }c_{e}<s_{i,+}^{r} \},\\
	&&V_{i,r}^{b}:  \{(\lambda_{i}^r, c_e)|\lambda_{i}^r\in (\mu_{0,i},\infty) \text{ and }s_{i,+}^r<c_{e}<c^{*}_{i,-} \},\\
	&&V_{i,r}^{c}: \hspace{-0.25cm}\begin{array}{ll}
		\left\{(\lambda_{i}^r, c_e)\,\begin{array}{|ll}
			\lambda_{1}^r\in (\mu_{0,i},\mu^{*}_{i,-}) \text{ and }c^{*}_{i,-}<c_{e}<k_{i}(\lambda_{i}^r), \\
	\text{or}~~ \lambda_{i}^r\in [\mu^{*}_{i,-},\infty)\text{ and }c_{e}>c^{*}_{i,-}
	\end{array}\right\},
	\end{array}\\
	&&V_{i,r}^{d}: \{(\lambda_{i}^r, c_e)|\lambda_{i}^r\in (0,\mu^{*}_{i,-}) \text{ and }c_{e}>c_{i,+}(\lambda_{i}^r)\vee k_{i}(\lambda_{i}^r) \}.
	\end{eqnarray*}
\end{lemma}

\begin{lemma}\label{lem:left-division}
For $i=1,2$, define the curves $\gamma_{i}^{o},\gamma_{i}^{p},\cdots, \gamma_{i}^{s}$ as follows:
	\begin{eqnarray*}
		&&	\gamma_{i}^{o}: c_{e}=-c_{i,-}(\lambda^{l}_{i}) \hspace{1.8em}\text{ for }\lambda_{i}^{l}\in (0,\mu^{*}_{i,-}],\\
		&&	\gamma_{i}^{p}: c_{e}\equiv -c^{*}_{i,-} \hspace{3.5em}\text{ for }\lambda^{l}_{i}\in [\mu^{*}_{i,+},\infty),\\
		&&	\gamma_{i}^{q}: c_{e}=-g_{i}(\lambda_{i}^{l})  \hspace{2.5em}\text{ for }\lambda_{i}^{l}\in (\mu^{*}_{i,+},L_{i}'(\bar{c}_{i})],\\
		&&	\gamma_{i}^{r}:c_{e}\equiv-\bar{c}_{i} \hspace{4.1em}\text{ for }\lambda_{i}^{l}\in [L_{i}'(\bar{c}_{i}),+\infty),\\
		&&	\gamma_{i}^{s}:c_{e}=-H_{i}'(\lambda^{r}_{i}) \hspace{2em}\text{ for }\lambda^{l}_{i}\in [\mu^{*}_{i,-},L_{i}'(\mu^{*}_{i,-}\bar{c}_{i})].
		\end{eqnarray*}
	Then the half-plane $E_{i}^l:=\{(\lambda_{i}^l, c_e)|\lambda_{i}^l\in(0,\infty),~c_e\in\mathbb R\}$ is divided into the following four disjoint regions:
\begin{eqnarray*}
	&& V_{i,l}^{a}: 	\{(\lambda_{i}^l, c_e)|\lambda_{i}^l\in (0,\infty) \text{ and }c_{e}>-s_{i,-}^{l} \},\\
	&&V_{i,l}^{b}:  \{(\lambda_{i}^l, c_e)|\lambda_{i}^r\in (\mu^{*}_{i,-},\infty) \text{ and } -H_{i}'(\lambda_{i}^l)\vee -\bar{c}_{i}<c_{e}<-c^{*}_{i,-} \},\\
	&&V_{i,l}^{c}: \hspace{-0.25cm}\begin{array}{ll}
		\left\{(\lambda_{i}^l, c_e)\,\begin{array}{|ll}
			\lambda_{i}^r\in (0,\mu^{*}_{i,+}] \text{ and }c_{e}<-c_{i,-}(\lambda_{i}^l), \\
	\text{or}~~ \lambda_{i}^l\in (\mu^{*}_{i,+},\mu^{*}_{i,-})\text{ and } -g_{i}(\lambda_{i}^l)<c_{e}<- c_{i,-}(\lambda_{i}^l),\\
	\text{or}~~\lambda_{i}^l\in [\mu^{*}_{i,-},L_{i}'(\bar{c}_{i})) \text{ and } -g_{i}(\lambda_{i}^l)<c_{e}<-H_{i}'(\lambda_{i}^l)
	\end{array}\right\},
	\end{array}\\
	&&V_{i,l}^{d}:  \hspace{-0.25cm}\begin{array}{ll}
		\left\{(\lambda_{i}^l, c_e)\,\begin{array}{|ll}
			\lambda_{i}^r\in (\mu^{*}_{i,+},L_{i}'(\bar{c}_{i})] \text{ and }c_{e}<-g_{i}(\lambda_{i}^l), \\
	\text{or}~~ \lambda_{i}^l\in (L_{i}'(\bar{c}_{i}),\infty)\text{ and } c_{e}<- \bar{c}_{i}\\
	\end{array}\right\}.
	\end{array}\\
	\end{eqnarray*}
\end{lemma}

\subsection{Viscosity solutions}

For a given function $R:\mathbb{R}\rightarrow\mathbb{R}$ which may be discontinuous, define its upper semi-continuous (u.s.c. for short) and lower semi-continuous (l.s.c. for short) envelope of $R$ as follows:
\begin{equation*}
	R^*(s)=\limsup_{s'\to s} R(s')\quad \text{ and } \quad R_*(s)=\liminf_{s'\to s} R(s').
\end{equation*}
Consider the following time-dependent Hamilton-Jacobi equation:
\begin{equation}\label{t-dep}
\min\left\{\partial_{t}w+H_{i}(\partial_{x}w)+R(x/t),w\right\}=0, \quad (x,t)\in\mathbb{R}\times\mathbb{R}^{+},
\end{equation}
where $H_{i}$ is given by \eqref{eq:hi}.

We introduce the definition of discontinuous viscosity solutions for \eqref{t-dep} as in \cite{barles2013}.
\begin{definition}
	{\rm
	A locally bounded l.s.c. function $w$ is  a  {\it viscosity supersolution} of \eqref{t-dep} if $w\geqslant  0$, and for any test functions $\phi\in C^1(\mathbb{R}\times\mathbb{R}^{+})$, if $(x_0,t_{0})\in \mathbb{R}\times\mathbb{R}^{+}$ is a strict local minimum point of $w-\phi$, then
\begin{equation*}
	\partial_t\phi(x_0,t_{0})+ H_{i}(\partial_x\phi(x_0,t_{0}))+R^{*}(x_0/t_{0})\geqslant  0.
\end{equation*}
	A locally bounded u.s.c. function $w$ is a {\it viscosity subsolution} of \eqref{t-dep} if for any test functions $\phi\in C^1(\mathbb{R}\times\mathbb{R}^{+})$,  if $(x_0,t_{0})\in \mathbb{R}\times\mathbb{R}^{+}$  is a strict local maximum point of $w-\phi$ and $w(x_0,t_{0})>0$, then
 \begin{equation*}
		\partial_t\phi(x_0,t_{0})+ H_{i}(\partial_x\phi(x_0,t_{0}))+R_{*}(x_0/t_{0})\leqslant  0.
	\end{equation*}
	}
\end{definition}

	Now consider the following time-independent Hamilton-Jacobi equation
\begin{equation}\label{t-indep}
	\min\left\{\rho(s)-s\rho'(s)+H_{i}(\rho'(s))+R(s),\rho(s) \right\}=0, \quad s\in\mathbb{R},
\end{equation}
where $H_{i}$ is given by \eqref{eq:hi}.

\begin{definition}
		{\rm
	A locally bounded l.s.c. function $\rho$ is a {\it viscosity supersolution} of \eqref{t-indep} if $\rho\geqslant0$, and for any test function $\phi\in C^{1}(\mathbb{R})$, if $s_0$ is a local minimum point of $\rho-\phi$, then
	\begin{equation}\label{eq:t-indep-vis-sup}
		 \rho(s_{0})-s_{0}\phi'(s_{0})+H_{i}(\phi'(s_{0}))+R^{*}(s_{0})\geqslant0.
	\end{equation}
	A locally bounded u.s.c. function $\rho$ is a {\it viscosity subsolution} of \eqref{t-indep} if for any test function $\phi\in C^{1}(\mathbb{R})$, if $s_0$ is a local maximum point of $\rho-\phi$ and $\rho(s_{0})>0$, then
	\begin{equation*}
		 \rho(s_{0})-s_{0}\phi'(s_{0})+H_i(\phi'(s_{0}))+R_{*}(s_{0})\leqslant0.
	\end{equation*}
		}
\end{definition}

\begin{lemma}\label{lem:samehj}
	Let $w(x,t) = t\rho\left({x}/{t}\right)$ for $(x,t)\in \mathbb{R}\times\mathbb{R}^{+}$.
Then $\rho(s)$ is a viscosity subsolution (resp. supersolution) of \eqref{t-indep} on  $\mathbb{R}$ if and only if $w(x,t)$ is  a viscosity subsolution (resp. supersolution) of \eqref{t-dep} on $\mathbb{R}\times\mathbb{R}^{+}$.
\end{lemma}

This lemma follows directly by adapting the proofs of \cite[Lemma 2.5]{lam2022a} or \cite[Lemma 2.3]{liu2021}, and we omit the details of its proof.

Now, we define an u.s.c. function and three l.s.c. functions, depending on the value of $c_{e}$, as follows:
\begin{equation}\label{eq:shift-bar}
	\bar{R}(s)=
\begin{cases}
	r_{1}\alpha_{-} &\text{for } s\in (-\infty,c_{e}],\\
	r_{1}\alpha_{+}  &\text{for } s\in (c_{e},\infty).
\end{cases}
\end{equation}

\begin{subequations}
	\begin{align}\label{eq:shift-under-1}
		\text{If } c_{e}>s_{2,-}^r, \hspace{2cm}	\underline R_{1}(s)=\left\{\begin{array}{ll}
			r_{1}\alpha_{-} &\text{for }s\in(-\infty,-s_{2,-}^l),\\
			r_{1}(	\alpha_{-}-a\mathcal{V}_{-}) &\text{for }s\in[-s_{2,-}^l,s_{2,-}^r],\\
			r_{1}\alpha_{-} &\text{for }s\in(s_{2,-}^r,c_{e}),\\
			r_{1}\alpha_{+} &\text{for }s\in[c_{e},\infty).
	\end{array}\right.
	\end{align}
\begin{align}\label{eq:shift-under-2}
	\hspace{0.1cm} \text{If } -s_{2,-}^l\leqslant c_{e}\leqslant s_{2,-}^r, \hspace{0.4cm} \underline  R_{2}(s)=\left\{\begin{array}{ll}
		r_{1}\alpha_{-} &\text{for }s\in(-\infty,-s_{2,-}^l),\\
		r_{1}(\alpha_{-}-a\mathcal{V}_{-}) &\text{for }s\in[-s_{2,-}^l,c_{e}),\\
		r_{1}	(\alpha_{+}-a\mathcal{V}_{-})  &\text{for }s\in[c_{e},s_{2,-}^r],\\
		r_{1}	\alpha_{+} &\text{for }s\in[s_{2,-}^r,\infty).
	\end{array}\right.
\end{align}
\begin{align}\label{eq:shift-under-3}
	\hspace{0.1cm}\text{If } c_{e}<-s_{2,-}^l,  \hspace{1.7cm} \underline 	 R_{3}(s)=\left\{\begin{array}{ll}
		r_{1}\alpha_{-} &\text{for }s\in(-\infty,c_{e}),\\
		r_{1}\alpha_{+} &\text{for }s\in[c_{e},-s_{2,-}^l),\\
		r_{1}(\alpha_{+}-a\mathcal{V}_{-})  &\text{for }s\in[-s_{2,-}^l,s_{2,-}^r],\\
		r_{1}\alpha_{+} &\text{for }s\in(s_{2,-}^r,\infty).
\end{array}\right.
\end{align}
	\end{subequations}
\begin{lemma}\label{lem:viscosity-comparison-t-indep}
Assume that $R:\mathbb{R}\rightarrow \mathbb{R}^{+}$ is defined  by \eqref{eq:shift-bar} or \eqref{eq:shift-under-1}-\eqref{eq:shift-under-3} or \eqref{eq:R0-usc} below. Let $\bar \rho$ (resp. $\underline \rho$) be a nonnegative viscosity supersolution (resp. subsolution) of \eqref{t-indep} on $\mathbb{R}$. Suppose further that
\begin{align*}
	\underline \rho(0)\leqslant \bar \rho(0), \quad \underset{s\rightarrow+\infty}\lim\, \underline \rho(s)=+\infty (resp. \underset{s\rightarrow-\infty}\lim\, \underline \rho(s)=+\infty),
\end{align*}
and
	\begin{equation*}
	\underset{s\rightarrow+\infty}\lim\, \frac{\underline \rho(s)}{s}\leqslant \underset{s\rightarrow+\infty}\liminf\,{\frac{\bar \rho(s)}{s}}\left(resp. \underset{s\rightarrow-\infty}\lim\, \frac{\underline \rho(s)}{|s|}\leqslant \underset{s\rightarrow-\infty}\liminf\,{\frac{\bar \rho(s)}{|s|}}\right).
	\end{equation*}
	If in addition
	\begin{align*}
		\underset{s\rightarrow+\infty}\lim\, \frac{\underline \rho(s)}{s}\leqslant +\infty (resp. \underset{s\rightarrow-\infty}\lim\, \frac{\underline \rho(s)}{|s|}\leqslant +\infty),
	\end{align*}
 then $\underline \rho(s)\leqslant \bar \rho(s)$ in $\mathbb{R}^{+}$(resp. $\mathbb{R}^{-}$).
	\end{lemma}

		This lemma is a direction consequence of {\cite[Proposition 2.4]{tao-shift}}; we omit the proof here.

\section{Upper and lower limiting Hamilton-Jacobi equations}\label{sec:hamilton-limit}
In this section, we derive the upper and lower limiting Hamilton-Jacobi equations under assumptions  (J), (A), (H1)-(H2) and (FU).

\subsection{Crucial estimates}
Define the hyperbolic scaling for the $u$-equation by
\begin{align}\label{eq:hyperbolic}
	 u^{\varepsilon}(x,t)=u\left(\frac{x}{\varepsilon},\frac{t}{\varepsilon}\right),  \quad (x,t)\in \mathbb{R}\times\mathbb{R}^{+}.
\end{align}
By the comparison principle for a scalar equation, the following Hopf-Cole transform is well-defined for any $\varepsilon>0$,
\begin{align}\label{eq:hopf}
	w^{\varepsilon}(x,t)=-\varepsilon\ln u^{\varepsilon}(x,t), \quad (x,t)\in \mathbb{R}\times\mathbb{R}^{+}.
\end{align}
A direct computation shows that $w^{\varepsilon}$ satisfies the following equation on $\mathbb{R}\times\mathbb{R}^{+}$
\begin{align}\label{eq:we-viscosity}
	 w^{\varepsilon}_t+d_{1}\left(\int_{\mathbb{R}}J_{1}(y)e^{-\frac{w^{\varepsilon}(x-\varepsilon y, t)-w^{\varepsilon}(x, t)}{\varepsilon}}\,dy-1\right)+r_{1}\left[\alpha\left(\frac{x-c_{e}t}{\varepsilon}\right)-u^{\varepsilon}-av\left(\frac{x}{\varepsilon},\frac{t}{\varepsilon}\right)\right]=0.
\end{align}

\begin{lemma}\label{lemma:local-bound}
 Assume that {\rm (I$^{\lambda}$)} holds for some $\lambda\in(0,\infty)$. Then there exists a constant $K>0$, independent of small $\varepsilon$,  such that
 \begin{equation}\label{eq:optimal-we}
	\max\left\{\lambda^{r}_{1}x_{+}+\lambda^{l}_{1}x_{-} -K(t+\varepsilon),-\varepsilon\ln \alpha_{-}\right\}\leqslant  w^{\varepsilon}(x,t)\leqslant \lambda^{r}_{1}x_{+}+\lambda^{l}_{1}x_{-} +K(t+\varepsilon)
 \end{equation}
 for all $ (x,t)\in\mathbb{R}\times\mathbb{R}^{+}$, where $x_{+}=x\vee 0$ and $x_{-}=-x\vee 0$ for $x\in\mathbb{R}$.
\end{lemma}
\begin{proof}
By   \eqref{initial-exp}, there exist two constants $C_{2}>C_{1}>0$ such that
	\begin{equation*}
C_{1}e^{-\left(\lambda^{r}_{1}x_{+}+\lambda^{l}_{1}x_{-}\right)}\leqslant u_{0}(x)\leqslant 	 C_{2}e^{-\left(\lambda^{r}_{1}x_{+}+\lambda^{l}_{1}x_{-}\right)} \quad \text{for~}x\in\mathbb{R}.
	\end{equation*}
	From the Hopf-Cole transform, it follows that
	\begin{equation*}\label{eq:w-epsilon-initial}
	 -\varepsilon\ln C_{2}+\lambda^{r}_{1}x_{+}+\lambda^{l}_{1}x_{-}\leqslant w^{\varepsilon}(x,0)\leqslant -\varepsilon\ln C_{1}+\lambda^{r}_{1}x_{+}+\lambda^{l}_{1}x_{-} \quad \text{for~}x\in \mathbb{R}.
	\end{equation*}

	Now, define
	\begin{equation*}\label{eq:sup-1}
		\bar w^{\varepsilon}(x,t)=\lambda^{r}_{1}x_{+}+\lambda^{l}_{1}x_{-}+K_{1}(t+\varepsilon),
	\end{equation*}
	where $K_{1}>0$ is chosen such that
	\begin{equation*}\label{eq:constant-K}
		K_{1}\geqslant \max\left\{d_{1}+\alpha_{-}+r_{1}(\alpha_{-}+a\mathcal{V}_{-}),|\ln C_{1}|\right\}.
	\end{equation*}
Using the comparison principle, one can verify that $w^{\varepsilon}\leqslant \bar w^{\varepsilon}$ on $\mathbb{R}\times\mathbb{R}^{+}$ for any $\varepsilon>0$.

Next, define
\begin{equation*}
	\underline w_{1}^\varepsilon(x,t)=\lambda^{r}_{1}x-K_{2}(t+\varepsilon)\quad \text{and}\quad \underline w_{2}^\varepsilon(x,t)=-\lambda^{l}_{1}x-K_{2}(t+\varepsilon),
\end{equation*}
with $K_{2}>0$ satisfying
\begin{align*}
		K_{2}\geqslant \max\left\{|\ln C_{2}|,d_{1}\int_{\mathbb{R}}J_{1}(y)e^{(\lambda^{r}_{1}\vee \lambda^{l}_{1})y}\,dy-d_{1}+r_{1}\alpha_{-}\right\}.
\end{align*}
By the comparison principle, for all small $\varepsilon>0$, we have
\begin{equation*}
	w^{\varepsilon} \geqslant \underline w_{1}^\varepsilon \quad \text{on~}\mathbb{R}\times\mathbb{R}^{+}\quad \text{and}\quad w^{\varepsilon} \geqslant \underline w_{2}^\varepsilon \quad \text{on~}\mathbb{R}\times\mathbb{R}^{+}.
\end{equation*}
Since $u^{\varepsilon}\leqslant \alpha_{-}$ on $\mathbb{R}\times\mathbb{R}^{+}$ for  any $\varepsilon>0$, it follows that $w^{\varepsilon}\geqslant -\varepsilon\ln \alpha_{-}$ on $\mathbb{R}\times\mathbb{R}^{+}$. Therefore,
\begin{equation*}\label{eq:sub-w-epsilon}
	w^{\varepsilon}(x,t) \geqslant \max\left\{\lambda^{r}_{1}x_{+}+\lambda^{l}_{1}x_{-} - K_{2}(t+\varepsilon),-\varepsilon\ln \alpha_{-}\right\}, \quad (x,t)\in\mathbb{R}\times\mathbb{R}^{+}.
\end{equation*}
The conclusion follows by setting $K=\max\{K_{1},K_{2}\}$.
\end{proof}

\begin{proposition}\label{prop:local-bound-u}
	Assume that {\rm (I$^{\lambda}$)} holds for some $\lambda\in(0,\infty]$.
 Then for any compact subset $Q$ of $\mathbb{R}\times\mathbb{R}^{+}$, there exist constants $C=C(Q)>0$ and $\varepsilon_{0}>0$ such that for any $\varepsilon\in(0,\varepsilon_{0})$,
\begin{equation}\label{eq:local-bound-w}
\underset{(x,t)\in Q}\sup |w^{\varepsilon}(x,t)|\leqslant C.
\end{equation}
Consequently, the following half-relaxed limits are well-defined:
\begin{equation}\label{eq:half-relaxed}
		w_*(x,t):=\liminf_{\substack{\varepsilon \rightarrow 0\\(y,s)\rightarrow(x,t)}}w^{\varepsilon}(y,s), \quad w^*(x,t):=\limsup_{\substack{\varepsilon \rightarrow 0\\(y,s)\rightarrow(x,t)}} w^{\varepsilon}(y,s).
		\end{equation}
Moreover, the following results hold in the classical sense:
\begin{itemize}
\item $w^{*}(0,t)=w_{*}(0,t)=0$ for all $t\in\mathbb{R}^{+}$,
\item if {\rm (I$^{\lambda}$)} holds with $\lambda<\infty$, then
\begin{align}\label{eq:viscosity-initial-value}
	  w^{*}(x,0)=w_{*}(x,0)=\left\{\begin{array}{ll}
		\lambda^{r}_{1}x &  \text{for~} x\in\mathbb{R}^{+}\cup\{0\},\\
	-\lambda^{l}_{1}x &  \text{for~}x\in\mathbb{R}^{-},
	\end{array}\right.
	\end{align}
\item	if {\rm (I$^{\infty}$)} holds, then
\begin{align}\label{eq:viscosity-initial-value-fast-decay}
	w^{*}(x,0)=w_{*}(x,0)=+\infty \quad \text{for all~}x\neq 0.
\end{align}
\end{itemize}
\end{proposition}

\begin{proof}[Sketch of proof]
The locally uniform bound \eqref{eq:local-bound-w} follows from assumption (H2) and an argument similar to the proof of \cite[Theorem 4.5]{liang2020a}, we omit the details here.

By (H2) and a standard comparison argument based on the spreading theory in \cite{shen2010,xu2021}, we find that for any $\eta_{0}\in(0,\tilde c^{*})$,
\begin{equation}\label{eq:u-x0-r1}
	\underset{t\rightarrow\infty}\liminf \underset{|x|\leqslant \eta_{0}t}\inf\, u(x,t)\geqslant \alpha_{+}-a\mathcal{V}_{-}>0,
\end{equation}
where
\begin{align}\label{eq:speed-homo-low-bound}
	\tilde c^{*}=\underset{\mu\in \mathbb{R}^{+}}\inf\,\frac{H_{1}(\mu)+r_{1}(\alpha_{+}-a\mathcal{V}_{-})}{\mu}>0.
\end{align}
Hence, for any $t>0$ and $|x|\leqslant \eta_{0}t$, we get that
\begin{equation*}
	\underset{\varepsilon\rightarrow0}\liminf\, u^{\varepsilon}(x,t)=\underset{\varepsilon\rightarrow0}\liminf\, u\left(\frac{x}{\varepsilon},\frac{t}{\varepsilon}\right)\geqslant \alpha_{+}-a\mathcal{V}_{-}>0,
\end{equation*}
which implies that for any given $t>0$,
\begin{align}\label{eq:we-bound-01}
	\underset{|x|\leqslant \eta_{0}t}\sup\,w^{\varepsilon}(x,t)\leqslant -\varepsilon\ln (\alpha_{+}-a\mathcal{V}_{-}).
\end{align}
Taking the half-relaxed limit as $\varepsilon\rightarrow0$ in \eqref{eq:we-bound-01} yields $w^{*}(0,t)\leqslant 0$ for any given $t>0$. On the other hand, since $u\leqslant \alpha_{-}$, we have  $w_{*}(0,t)\geqslant \lim_{\varepsilon\rightarrow0}\, -\varepsilon \ln \alpha_{-}= 0$ for $t>0$. Hence, $w^{*}(0,t)=w_{*}(0,t)=0$ for any $t>0$.

We now prove \eqref{eq:viscosity-initial-value}. From Lemma \ref{lemma:local-bound}, we have
	\begin{align*}
		\max\left\{\lambda^{r}_{1}x_{+}+\lambda^{l}_{1}x_{-} -Kt,0\right\}\leqslant w_{*}(x,t)\leqslant w^{*}(x,t)\leqslant \lambda^{r}_{1}x_{+}+\lambda^{l}_{1}x_{-} +Kt, \quad (x,t)\in\mathbb{R}\times\mathbb{R}^{+}.
	\end{align*}
Letting $t\rightarrow 0^{+}$ and using the definition of the half-relaxed limit, we obtain  \eqref{eq:viscosity-initial-value}.

Finally, we show \eqref{eq:viscosity-initial-value-fast-decay}.  If (I$^{\infty}$) holds, then for any $\lambda>0$, there exists $C_{3}>\max\{1, \alpha_{-}\}$  such that $u_{0}(x)\leqslant C_{3}\exp{(-\lambda|x|)}$ for all $x\in\mathbb{R}$. Let $\bar{u}$ be the solution of
\begin{equation*}
\begin{cases}
	\bar u_{t}=d_{1}(J_{1}\ast \bar u-\bar u)+r_{1}\bar u(\alpha_{-}-\bar u), &(x,t)\in \mathbb{R}\times\mathbb{R}^{+}, \\
	\bar u(x,0)= C_{3}\exp{(-\lambda|x|)},& x\in \mathbb{R}.
\end{cases}
\end{equation*}
By the comparison principle, we have
	\begin{equation}\label{eq:sup-u}
		u\leqslant \bar{u}\quad \text{on~}\mathbb{R}\times[0,\infty).
	\end{equation}
Repeating the argument in Lemma \ref{lemma:local-bound}, there exists $K_{3}>0$, independent of $\varepsilon$, such that
	\begin{equation*}
			\bar{w}^{\varepsilon}(x,t)\geqslant \max\left\{\lambda|x| -K_{3}(t+\varepsilon),-\varepsilon\ln C_{3}\right\}, \quad (x,t)\in \mathbb{R}\times\mathbb{R}^{+},
	\end{equation*}
	where $\bar{w}^{\varepsilon}(x,t)=-\varepsilon\ln \bar{u}(x/\varepsilon,t/\varepsilon)$. Employing \eqref{eq:sup-u} and letting $\varepsilon\rightarrow0^{+}$, we obtain
	\begin{equation*}
			w_{*}\geqslant \max\left\{\lambda|x|-K_{3}t,0\right\}, \quad (x,t)\in \mathbb{R}\times\mathbb{R}^{+}.
	\end{equation*}
	For $x\neq0$, letting $t\rightarrow0^{+}$, we have
	\begin{equation}\label{eq:w*-to-infinity}
			w_{*}(x,0)\geqslant \lambda|x|, \quad x\in \mathbb{R}.
	\end{equation}
Since $\lambda>0$ is arbitrary, taking $\lambda\rightarrow+\infty$ yields $w_{*}(x,0)=+\infty$ for any  $x\neq0$. Hence,  $w^{*}(x,0)\geqslant w_{*}(x,0)=+\infty$ for any  $x\neq0$. The proof is complete.
\end{proof}

\subsection{Lower limiting Hamilton-Jacobi equation}
\begin{proposition}\label{prop:lower-hj-control}
	Assume that {\rm (I$^{\lambda}$)} holds for some $\lambda\in(0,\infty]$. Then $w_{*}$ is a viscosity supersolution of the following Hamilton-Jacobi equation
	\begin{align}\label{eq:hj-tdep-up}
		\min\left\{\partial_{t}w+H_{1}(\partial_{x}w)+\bar R(x/t),w\right\}=0, \quad (x,t)\in \mathbb{R}\times\mathbb{R}^{+},
	\end{align}
	where $H_{1}$ and $\bar R$ are given by \eqref{eq:hi} and \eqref{eq:shift-bar}, respectively.
\end{proposition}

\begin{proof}
Since $w^{\varepsilon}\geqslant -\varepsilon\ln \alpha_{-}$ on $\mathbb{R}\times\mathbb{R}^{+}$, it follows that $w_{*}\geqslant 0$ on the same domain. To complete the proof, for any test function $\phi\in C^{1}(\mathbb{R}\times\mathbb{R}^{+})$, suppose that $w_{*}-\phi$ attains its strict local minimum at $(x_{0},t_{0})\in \mathbb{R}\times\mathbb{R}^{+}$. It suffices to show that
	\begin{align*}
		\partial_{t}\phi(x_{0},t_{0})+H_{1}(\partial_{x}\phi(x_{0},t_{0}))+\bar R^{*}(x_{0}/t_{0})\geqslant 0.
	\end{align*}
	By the definition of $w_{*}$, there exist sufficiently small $r>0$ and $x_1>0$, and two  sequences $\{\varepsilon_n\}$ and $\{(x_{\varepsilon_n},t_{\varepsilon_n})\}$ with $|x_{\varepsilon_n}-x_0|\leqslant  x_1$ and $|t_{\varepsilon_n}-t_0|\leqslant  r$ such   that $\varepsilon_n\rightarrow0$ and $(x_{\varepsilon_n},t_{\varepsilon_n})\rightarrow(x_0,t_0)$ as $n\rightarrow\infty$,  and  $w^{\varepsilon_n}-\phi$ attains its global minimum in $Q := (x_0-2x_1,x_0+2x_1)\times(t_0-r,t_0+r)$ at $(x_{\varepsilon_n},t_{\varepsilon_n})$. Particularly, it holds
\begin{equation}\label{eq:we-QMr}
	 (w^{\varepsilon_n}-\phi)(x_{\varepsilon_n},t_{\varepsilon_n})\leqslant (w^{\varepsilon_n}-\phi)(x,t),\quad  \forall (x,t)\in Q.
\end{equation}
Since $w^{\varepsilon_{n}}$ solves \eqref{eq:we-viscosity} and $u,\,v$ are nonnegative, a direct computation shows that at $(x_{\varepsilon_n},t_{\varepsilon_n})$,
\begin{equation}\label{eq:phi-e-partial}
	\partial _t \phi+d_{1}\displaystyle\int_\mathbb{R}J_{1}(y)\left(e^{-\frac{w^{\varepsilon_n}\left(x_{\varepsilon_n}-{\varepsilon_n} y,t_{\varepsilon_n}\right)-w^{\varepsilon_n}(x_{\varepsilon_n},t_{\varepsilon_n})}{{\varepsilon_n}}}-1\right)\,dy+r_{1}\alpha\left(\frac{x_{\varepsilon_n}-c_{e}t_{\varepsilon_n}}{\varepsilon_{n}}\right)\geqslant0.
\end{equation}
Choose $M>0$ such that $\text{spt}(J)\subset [-M,M]$. For sufficiently large  $n$, we have $|x_{\varepsilon_n}-x_0|<x_1$ and   $|\varepsilon_n y|<x_1$ for $y\in[-M,M]$, and then $x_{\varepsilon_n}-{\varepsilon_n} y\in (x_0-2x_1,x_0+2x_1)$. Therefore, we obtain from \eqref{eq:we-QMr} that
\begin{equation}\label{eq:I-Me}
	\begin{aligned}
		 &\int_\mathbb{R}J_{1}(y)\left(e^{-\frac{w^{\varepsilon_n}\left(x_{\varepsilon_n}-{\varepsilon_n} y,t_{\varepsilon_n}\right)-w^{\varepsilon_n}(x_{\varepsilon_n},t_{\varepsilon_n})}{{\varepsilon_n}}}\right)\,dy\\
		 \leqslant&\displaystyle\int_{-M}^{M}J_{1}(y)\left(e^{-\frac{\phi\left(x_{\varepsilon_n}-{\varepsilon_n} y,t_{\varepsilon_n}\right)-\phi\left(x_{\varepsilon_n},t_{\varepsilon_n}\right)}{{\varepsilon_n}}}\right)\,dy:=I^{M,{\varepsilon_n}}.
	\end{aligned}
\end{equation}
Since $\phi\in C^1$ and $(x_{\varepsilon_n},t_{\varepsilon_n})\rightarrow(x_0,t_0)$ as $n\rightarrow\infty$. It holds uniformly  for $y\in[-M,M]$ that
\begin{equation*}
	 \lim_{n\to\infty}\frac{\phi\left(x_{\varepsilon_n}-{\varepsilon_n} y,t_{\varepsilon_n}\right)-\phi\left(x_{\varepsilon_n},t_{\varepsilon_n}\right)}{{\varepsilon_n}}=-y\partial_{x} \phi(x_0,t_0),
\end{equation*}
and hence,
 \begin{equation}\label{eq:local-hp}
	 \lim_{n\to\infty}I^{M,\varepsilon_n}=\displaystyle\int_{[-M,M]}J_{1}(y) e ^{y\partial_{x}  \phi(x_0,t_0)}\,dy.
 \end{equation}
Moreover, since $(x_{\varepsilon_n},t_{\varepsilon_n})\rightarrow(x_0,t_0)$ as $n\to\infty$, we have
\begin{equation}\label{eq:r-usc}
	 \underset{n\rightarrow\infty}\limsup\,\,\alpha\left(\frac{x_{\varepsilon_n}-c_{e}t_{\varepsilon_n}}{\varepsilon_n}\right)\leqslant \bar R\left(\frac{x_{0}}{t_{0}}\right)=\bar{R}^{*}\left(\frac{x_{0}}{t_{0}}\right),
\end{equation}
where the last equality follows from the fact that $\bar{R}$ is u.s.c. on $\mathbb{R}$.
Passing  to the limit as $n\rightarrow\infty$ in  \eqref{eq:phi-e-partial} and using \eqref{eq:I-Me}-\eqref{eq:r-usc}, we complete  the proof.
\end{proof}

\begin{lemma}\label{lem:low-hj-control}
	Assume that {\rm (I$^{\lambda}$)} holds for some $\lambda\in(0,\infty]$. Then $\rho_{*}(x/t):=w_{*}\left({x}/{t},1\right)$ is a viscosity supersolution of the following Hamilton-Jacobi equation
	\begin{equation}\label{eq:hj-tindep-up}
		\min\left\{\rho-s\rho'+H_{1}(\rho')+\bar R(s),\rho\right\}=0, \quad s\in \mathbb{R}^{+}~(\text{or}\,~\mathbb{R}^{-}),\\
	\end{equation}
	where $H_{1}$ and $\bar R$ are given by \eqref{eq:hi} and \eqref{eq:shift-bar}, respectively. Moreover, the following boundary condition holds in the classical sense:
	\begin{align}\label{eq:hj-tindep-up-boundary}
		\rho_{*}(0)=0, \quad \text{and} \quad \underset{s\rightarrow+\infty}\lim\, \rho_{*}(s)/s\geqslant \lambda_{1}^r~(or\, \underset{s\rightarrow-\infty}\lim\, \rho_{*}(s)/|s|\geqslant \lambda_{1}^l).
	\end{align}
\end{lemma}
\begin{proof}
The definition of the half-relaxed limit and \eqref{eq:hopf} imply that for $(x,t)\in \mathbb{R}\times\mathbb{R}^{+}$,
\begin{align*}
	w_{*}(x,t)=	\liminf_{\substack{\varepsilon \rightarrow 0\\(y,s)\rightarrow(x,t)}}-\varepsilon\ln u\left(\frac{y}{\varepsilon},\frac{s}{\varepsilon}\right)=t\cdot\liminf_{\substack{\varepsilon \rightarrow 0\\(y,s)\rightarrow(x/t,1)}}-\frac{\varepsilon}{t}\ln  u\left(\frac{y}{\varepsilon/t},\frac{s}{\varepsilon/t}\right)=tw_{*}\left(\frac{x}{t},1\right).
\end{align*}
By Lemma \ref{lem:samehj} and Proposition \ref{prop:lower-hj-control},   $\rho_{*}$ is indeed a viscosity supersolution of \eqref{eq:hj-tindep-up} on $\mathbb{R}^{+}$. It remains to verify the boundary conditions.
From Proposition \ref{prop:local-bound-u}, we obtain $\rho_{*}(0)=w_{*}(0,1)=0$. Furthermore, since $w_{*}$ is l.s.c., we have
\begin{align*}
	&\underset{s\rightarrow+\infty}\liminf\, \frac{\rho_{*}(s)}{s}=\underset{s\rightarrow+\infty}\liminf\, w_{*}\left(1,\frac{1}{s}\right)\geqslant w_{*}(1,0)=\lambda_{1}^{r}\in(0,\infty],\\
	&\underset{s\rightarrow-\infty}\liminf\, \frac{\rho_{*}(s)}{|s|}=\underset{s\rightarrow-\infty}\liminf\, w_{*}\left(-1,\frac{1}{-s}\right)\geqslant w_{*}(-1,0)=\lambda_{1}^{l}\in(0,\infty].
\end{align*}
This completes the proof.
\end{proof}

\subsection{Upper limiting Hamilton-Jacobi equation}
\begin{proposition}\label{prop:upper-hj-control}
	Assume that {\rm (I$^{\lambda}$)} holds for some $\lambda\in(0,\infty]$. Then $w^{*}$ is a viscosity subsolution of the following Hamilton-Jacobi equation
	\begin{align}\label{eq:hj-tdep-low}
		\min\left\{\partial_{t}w+H_{1}(\partial_{x}w)+\underline R(x/t),w\right\}=0, \quad (x,t)\in \mathbb{R}\times\mathbb{R}^{+},
	\end{align}
	where $H_{1}$ is given by \eqref{eq:hi}, and  $\underline R$ is given by \eqref{eq:shift-under-1}, \eqref{eq:shift-under-1} or \eqref{eq:shift-under-3}, depending on the value of $c_{e}$.
\end{proposition}
\begin{proof}[Sketch of proof]
	For any test function $\phi\in C^1(\mathbb{R}\times\mathbb{R}^+)$, suppose that $w^{*}-\phi$ attains its strict local maximum at $(x_0,t_0)\in \mathbb{R}\times\mathbb{R}^{+}$ and $w^*(x_{0},t_{0})>0$, we aim to show that
\begin{equation}\label{eq:phi-viscosity-sub}
	\partial _t \phi(x_0,t_0)+H_{1}(\partial_{x}\phi(x_0,t_0))+\underline{R}_{*}(x_{0}/t_{0})\leqslant 0.
\end{equation}
By the definition of $w^{*}$, there exist sequences $\{\varepsilon_n\}$ and $\{(x_{\varepsilon_n},t_{\varepsilon_n})\}$ such that $\varepsilon_n\rightarrow 0$, $(x_{\varepsilon_n},t_{\varepsilon_n})\rightarrow (x_{0},t_{0})$, $w^{\varepsilon_{n}}(x_{\varepsilon_{n}},{t_{\varepsilon_{n}}})\rightarrow w^{*}(x_{0},t_{0})>0$ as $n\rightarrow\infty$, and $w^{\varepsilon_{n}}-\phi$ attains its maximum at $(x_{\varepsilon_{n}},{t_{\varepsilon_{n}}})$ in a small neighborhood of $(x_{0},t_{0})$. Choose a sufficiently large $M>0$ such that $\text{spt}(J)\subset [-M,M]$, we obtain from \eqref{eq:we-viscosity} that
\begin{equation}\label{eq:phi-e-partial-b}
\begin{aligned}
	\partial _t \phi(x_{\varepsilon_n},t_{\varepsilon_n})+&d_{1}\displaystyle\int_{-M}^{M}J_{1}(y)\left(e^{-\frac{\phi\left(x_{\varepsilon_n}-{\varepsilon_n} y,t_{\varepsilon_n}\right)-\phi(x_{\varepsilon_n},t_{\varepsilon_n})}{{\varepsilon_n}}}-1\right)\,dy\\
	 +&r_{1}\left[\alpha\left(\frac{x_{\varepsilon_n}-c_{e}t_{\varepsilon_n}}{\varepsilon_{n}}\right)-av\left(\frac{x_{\varepsilon_n}}{\varepsilon_{n}},\frac{t_{\varepsilon_n}}{\varepsilon_{n}}\right)\right]-r_{1}u^{\varepsilon_{n}}(x_{\varepsilon_n},t_{\varepsilon_n})\leqslant 0.
\end{aligned}
\end{equation}
Since  $(x_{\varepsilon_n},t_{\varepsilon_n})\rightarrow(x_0,t_0)$ as $n\rightarrow \infty$, we have
\begin{equation}\label{eq:r-lsc}
	 \underset{n\rightarrow \infty}\liminf\,\,\alpha\left(\frac{x_{\varepsilon_n}-c_{e}t_{\varepsilon_n}}{\varepsilon_n}\right)\geqslant \bar{R}_{*}\left(\frac{x_{0}}{t_{0}}\right)=\begin{cases}
			\alpha_{-}& \text{if }x_{0}/t_{0}<c_{e},\\
			\alpha_{+}& \text{if }x_{0}/t_{0}\geqslant c_{e}.
		\end{cases}
\end{equation}
Moreover, by the comparison principle and the spreading theory of the nonlocal dispersal equation, we obtain
\begin{align}\label{eq:prop-v}
	\quad 	\underset{n\rightarrow \infty}\limsup\,v\left(\frac{x_{\varepsilon_n}}{\varepsilon_{n}},\frac{t_{\varepsilon_n}}{\varepsilon_{n}}\right)\leqslant \begin{cases}
		\mathcal{V}_{-}& \text{if } -s_{2,-}^{l}\leqslant x_{0}/t_{0}\leqslant s_{2,-}^{r},\\
		0& \text{if }x_{0}/t_{0}<-s_{2,-}^{l}, \text{or } x_{0}/t_{0}>s_{2,-}^{r}.
	\end{cases}
\end{align}
Note that $u^{\varepsilon_{n}}(x_{\varepsilon_{n}},{t_{\varepsilon_{n}}})\rightarrow0$ as $n\rightarrow\infty$, the desired inequality \eqref{eq:phi-viscosity-sub} follows from \eqref{eq:phi-e-partial-b}-\eqref{eq:prop-v}  via a similar  argument to the proof of Proposition \ref{prop:lower-hj-control}.
\end{proof}

Similar to Lemma \ref{lem:low-hj-control}, we obtain  the  following corollary.
\begin{corollary}\label{coro:up-hj-control}
	Assume that {\rm (I$^{\lambda}$)} holds for some $\lambda\in(0,\infty]$. Then $\rho^{*}(x/t):=w^{*}\left({x}/{t},1\right)$ is a  viscosity subsolution of the following Hamilton-Jacobi equation
	\begin{equation}\label{eq:hj-tindep-low}
		\min\left\{\rho-s\rho'+H_{1}(\rho')+\underline R(s),\rho\right\}=0, \quad  s\in \mathbb{R}^{+}~(or\, \mathbb{R}^{-}),
	\end{equation}
where $H_{1}$ is given by \eqref{eq:hi},
and  $\underline R$ is given by \eqref{eq:shift-under-1}, \eqref{eq:shift-under-1} or \eqref{eq:shift-under-3}, depending on the value of $c_{e}$.
Moreover, the following boundary conditions hold in the classical sense:
	\begin{align}\label{eq:hj-tindep-low-boundary}
		\rho^{*}(0)=0, \quad \text{and} \quad \underset{s\rightarrow+\infty}\lim\, \rho^{*}(s)/s\leqslant \lambda_{1}^r~(or\, \underset{s\rightarrow-\infty}\lim\, \rho^{*}(s)/|s|\leqslant \lambda_{1}^l).
	\end{align}
\end{corollary}

\begin{lemma}\label{lem:persist-solu}
	Assume that {\rm (I$^{\lambda}$)} holds for some $\lambda\in(0,\infty]$ and the set $\mathcal{B}={\rm{Int}}\{(x,t)\in \mathbb{R}\times\mathbb{R}^{+}|w^{*}(x,t)=0\}$ is nonempty. Then
	\begin{align*}
		\underset{\varepsilon\rightarrow0}\liminf\,u^{\varepsilon}>0 \quad  \text{locally uniformly in } \, \mathcal{B}.
	\end{align*}
\end{lemma}
\begin{proof}
		Fix	 $(x_{0},t_{0})\in \mathcal{B}$, and	define the test function
	\begin{equation*}
		\phi(x,t)=|x-x_{0}|^{2}+|t-t_{0}|^{2}, \quad (x,t)\in\mathbb{R}^{+}\times\mathbb{R}^{+}.
\end{equation*}
Then $w^{*}-\phi$ attains its strict local maximum at $(x_{0},t_{0})$. By the definition of $w^{*}$, up to extraction of a subsequence, there exists a sequence $\{(x_{\varepsilon},t_{\varepsilon})\}_{\varepsilon}$ such that $(x_{\varepsilon},t_{\varepsilon})$ is a local maximum point of $w^{\varepsilon}-\phi$ and $(x_{\varepsilon},t_{\varepsilon})\rightarrow(x_{0},t_{0})$ as $\varepsilon\rightarrow0$. In particular, we have
\begin{equation}\label{eq:local-max}
 (w^{\varepsilon}-\phi)(x,t)\leqslant (	 w^{\varepsilon}-\phi)(x_{\varepsilon},t_{\varepsilon})
\end{equation}
in a small neighborhood of $(x_{0},t_{0})$. Since $w^{\varepsilon}$ satisfies \eqref{eq:we-viscosity}, it follows from \eqref{eq:local-max} that
\begin{equation}\label{eq:local-estimate-persist}
	\begin{aligned}
	 &r_{1}\left[u^\varepsilon(x_{\varepsilon},t_{\varepsilon})-\alpha\left(\frac{x_{\varepsilon}-c_{e}t_{\varepsilon}}{\varepsilon}\right)+av\left(\frac{x_{\varepsilon}}{\varepsilon},\frac{t_{\varepsilon}}{\varepsilon}\right)\right]\\
	\geqslant 	 &\partial_{t}\phi(x_{\varepsilon},t_{\varepsilon})-d_{1}\displaystyle\int_\mathbb{R}J_{1}(y)\left(1-e^{-\frac{\phi\left(x_{\varepsilon}-\varepsilon y,t_{\varepsilon}\right)-\phi(x_{\varepsilon},t_{\varepsilon})}{\varepsilon}}\right)\,dy.
\end{aligned}
\end{equation}
Noting that $\partial_{x}\phi(x_{0},t_{0})=0$, we apply a similar argument as in \eqref{eq:local-hp}  to obtain
\begin{equation}\label{eq:local-estimate-2}
		 \underset{\varepsilon\rightarrow0}\lim\,\displaystyle\int_\mathbb{R}J_{1}(y)\left(1-e^{-\frac{\phi\left(x_{\varepsilon}-\varepsilon y,t_{\varepsilon}\right)-\phi(x_{\varepsilon},t_{\varepsilon})}{\varepsilon}}\right)\,dy=0.
\end{equation}
Combining \eqref{eq:local-estimate-persist} and \eqref{eq:local-estimate-2}, we have
\begin{equation}\label{eq:local-estimate-3}
	\underset{\varepsilon\rightarrow0}\liminf\, u^\varepsilon(x_{\varepsilon},t_{\varepsilon})\geqslant \underset{\varepsilon\rightarrow0}\liminf\,\left[\alpha\left(\frac{x_{\varepsilon}-c_{e}t_{\varepsilon}}{\varepsilon}\right)-av\left(\frac{x_{\varepsilon}}{\varepsilon},\frac{t_{\varepsilon}}{\varepsilon}\right)\right].
\end{equation}
Moreover, by \eqref{eq:local-max} and the nonnegativity of $\phi$, we obtain $ w^{\varepsilon}(x_{0},t_{0})\leqslant w^{\varepsilon}(x_{\varepsilon},t_{\varepsilon})$. Then the Hopf-Cole transform implies that
	\begin{equation}\label{eq:local-estimate-3.2}
  u^\varepsilon(x_0,t_0)\geqslant u^\varepsilon(x_\varepsilon,t_\varepsilon).
	\end{equation}
Therefore, from \eqref{eq:local-estimate-3} and \eqref{eq:local-estimate-3.2}, it holds that
\begin{equation*}
	\liminf_{\varepsilon\to0} u^\varepsilon(x_0,t_0)\geqslant  \underset{\varepsilon\rightarrow0}\liminf\,\left[\alpha\left(\frac{x_{\varepsilon}-c_{e}t_{\varepsilon}}{\varepsilon}\right)-av\left(\frac{x_{\varepsilon}}{\varepsilon},\frac{t_{\varepsilon}}{\varepsilon}\right)\right] \geqslant \alpha_{+}-a\mathcal{V}_{-}>0.
\end{equation*}
Since  $(x_{0},t_{0})\in \mathcal{B}$ is arbitrary, the result is established.
\end{proof}
\begin{corollary}\label{coro:u-approach-max-capacity}
	Assume that {\rm (I$^{\lambda}$)} holds for some  $\lambda\in(0,\infty]$ and the conditions in Lemma \ref{lem:persist-solu} hold. Suppose further that $c_{e}>s_{2,-}^{r}$ or $c_{e}<-s_{2,-}^{l}$. Then
	\begin{align}
		 &\underset{\varepsilon\rightarrow0}\liminf\,u^{\varepsilon}\geqslant \alpha_{+} \quad  \text{locally uniformly in }
		\mathcal{B}_{1} =\left\{\begin{array}{ll}
			\mathcal{B}\cap\{(x,t)|c_{e}<x/t<-s_{2,-}^l\}\text{ or, } \\
			\mathcal{B}\cap\{(x,t)|x/t>s_{2,-}^r\}, 	\,\,  \text{if } c_{e}<-s_{2,-}^{l}; \\
			\mathcal{B}\cap\{(x,t)|x/t>c_{e}\}, 	\,\,  \text{if } c_{e}>s_{2,-}^{r}.
		\end{array}\right. \label{eq:u-terrace-min}\\
		&	\underset{\varepsilon\rightarrow0}\liminf\,u^{\varepsilon}= \alpha_{-} \quad  \text{locally uniformly in }
		\mathcal{B}_{2} =\left\{\begin{array}{ll}
				\mathcal{B}\cap\{(x,t) |s_{2,-}^{r}<x/t<c_{e}\} \text{ or, }\\
				\mathcal{B}\cap\{(x,t) |x/t<-s_{2,-}^l\}, \,\,\text{if } c_{e}>s_{2,-}^{r};\\
				\mathcal{B}\cap\{(x,t)|x/t<c_{e}\}, 	\,\,  \text{if } c_{e}<-s_{2,-}^{l}.
			\end{array}\right. \label{eq:u-terrace-max}
			\end{align}
\end{corollary}

\begin{proof}

We first prove \eqref{eq:u-terrace-min} only for the case $c_{e}<-s_{2,-}^{l}$, and the case $c_{e}>s_{2,-}^r$ is similar. For any $(x_{0},t_{0})$ belongs to either $\mathcal{B}\cap\{(x,t)|c_{e}<x/t<-s_{2,-}^l\}$ or $\mathcal{B}\cap\{(x,t)|x/t>s_{2,-}^r\}$, it follows from (A) and \eqref{eq:prop-v} that
	\begin{align}\label{eq:low-bound-a-v}
		 \underset{\varepsilon\rightarrow0}\liminf\,\left[\alpha\left(\frac{x_{0}-c_{e}t_{0}}{\varepsilon}\right)-av\left(\frac{x_{0}}{\varepsilon},\frac{t_{0}}{\varepsilon}\right)\right]= \alpha_{+}.
	\end{align}
	Then, by the same argument as in the proof of Lemma \ref{lem:persist-solu}, we have
	\begin{align*}
		 \underset{\varepsilon\rightarrow0}\liminf\,u^{\varepsilon}(x_{0},t_{0})\geqslant \alpha_{+}.
	\end{align*}

We now prove \eqref{eq:u-terrace-max} for the case $c_{e}>s_{2,-}^{r}$, the other case can be treated similarly. For any $(x_{0},t_{0})\in\mathcal{B}_{2}$, the left side of the equality  \eqref{eq:low-bound-a-v} is bounded below by $\alpha_{-}$. Since $u^{\varepsilon}\leqslant \alpha_{-}$ on $\mathbb{R}\times\mathbb{R}^{+}$, the result follows by the same reasoning as before.
\end{proof}

\section{Explicit formulas for the spreading speed of the prey}\label{sec:formulas-prey}
This section is devoted to deriving the explicit formulas of the spreading speed for the prey. Throughout, we assume that assumptions (J), (A), (H1)-(H2) and (FU) hold.

\subsection{Determinacy of the spreading speed}

\begin{proposition}\label{prop:speed-r}
Assume that {\rm (I$^{\lambda}$)} holds for some $\lambda\in(0,\infty]$,
and let  $\rho\in C([0,\infty);[0,\infty))$  satisfy  the boundary conditions
	\begin{align}\label{eq:vis-boundary-condition}
		\rho(0)=0, \quad  \underset{s\rightarrow+\infty}\lim\, \rho(s)/s=\lambda_{1}^r\in(0,\infty].
	\end{align}
If $\rho$ is a viscosity subsolution of \eqref{eq:hj-tindep-up} and a viscosity supersolution of \eqref{eq:hj-tindep-low} on $\mathbb{R}^{+}$,  and if there exists a unique $s_{r}>0$ such that
	\begin{align*}
		\rho(s)=0 \quad \text{for }s\in[0,s_{r}],\quad \text{and} \quad  \rho(s)>0 \quad \text{for }s\in(s_{r},\infty),
	\end{align*}
	then $c_{u}^r=s_{r}$.
\end{proposition}

\begin{proof}
 We first show that ${c}_{u}^r\leqslant s_{r}$.
Recall that $\rho_{*}$ given by Lemma \ref{lem:low-hj-control} is a viscosity supersolution of \eqref{eq:hj-tindep-up} and satisfies the boundary conditions \eqref{eq:hj-tindep-up-boundary}. Since $\rho$ is a viscosity subsolution of \eqref{eq:hj-tindep-up} and satisfies \eqref{eq:vis-boundary-condition}, the comparison principle (Lemma \ref{lem:viscosity-comparison-t-indep}) implies
\begin{align}\label{eq:vis-r-sub-control}
	\rho \leqslant \rho_{*} \quad \text{on }\mathbb{R}^{+}.
\end{align}
Thus,  $\rho_{*}(s)>0$ for any  $s>s_{r}$. We now claim that
\begin{align*}
	\underset{\varepsilon\rightarrow0}\limsup\,u^{\varepsilon}=0 \quad  \text{locally uniformly in }\{(x,t)\in\mathbb{R}^{+}\times\mathbb{R}^{+}|x/t>s_{r}\}.
\end{align*}
To verify this, fix $(x_{0},t_{0})$ such that $x_0/t_{0}>s_{r}$. Since $\rho_{*}(x_{0}/t_{0})>0$, Lemma \ref{lem:samehj} implies that $w_{*}(x_{0},t_{0})>\kappa$ for some $\kappa>0$. By the definition of the half-relaxed limit, $w^{\varepsilon}>\kappa/2$ in a small neighborhood of $(x_{0},t_{0})$ for any sufficiently small $\varepsilon$. The Hopf-Cole transform then implies that
\begin{align*}
	 u^{\varepsilon}(x,t)=e^{-\frac{w^{\varepsilon}(x,t)}{\varepsilon}}\leqslant e^{-\frac{\kappa}{2\varepsilon}} \quad \text{in a small neighborhood of } (x_{0},t_{0}).
\end{align*}
By the nonnegativity of $u^{\varepsilon}$, taking $\varepsilon\rightarrow0$ in the above inequality yields the claim.
Therefore, for any given $\eta_{2}>\eta_{1}>0$,  we have
\begin{align}\label{speed-dete-1}
		 \underset{t\rightarrow\infty}\lim\,\underset{(s_{r}+\eta_{1})t\leqslant x\leqslant (s_{r}+\eta_{2})t }\sup\,u(x,t)=\underset{\varepsilon\rightarrow0}\lim\, \underset{s_{r}+\eta_{1}\leqslant y\leqslant s_{r}+\eta_{2}}\sup\,u^{\varepsilon}(y,1)=0.
\end{align}
On the other hand, a standard comparison argument shows that
\begin{align}\label{speed-dete-2}
	\underset{t\rightarrow\infty}\lim\,\underset{x\geqslant (s_{1,-}^{r}+\eta)t}\sup\,u(x,t)=0 \quad \text{for any }\eta>0,
\end{align}
where $s_{1,-}^r$ is defined in \eqref{eq:speed-homo}. Combining \eqref{speed-dete-1} and \eqref{speed-dete-2}, we conclude that  ${c}_{u}^r\leqslant s_{r}$.

We now prove that ${c}_{u}^r\geqslant s_{r}$.
Recall that $\rho^{*}$ given by Corollary \ref{coro:up-hj-control} is a viscosity subsolution of \eqref{eq:hj-tindep-low} and satisfies the boundary conditions \eqref{eq:hj-tindep-low-boundary}. Since $\rho$ is a viscosity supersolution of \eqref{eq:hj-tindep-low} and satisfies \eqref{eq:vis-boundary-condition}, the comparison principle implies
\begin{align}\label{eq:vis-r-sup-control}
	\rho \geqslant \rho^{*} \quad \text{on }\mathbb{R}^{+}.
\end{align}
Thus, $\rho^{*}(s)=0$ for any  $ s\in[0,s_{r}]$. By Lemmas \ref{lem:samehj} and \ref{lem:persist-solu}, it follows that
\begin{align*}
	\underset{\varepsilon\rightarrow0}\liminf\,u^{\varepsilon}>0 \quad  \text{locally uniformly in }\{(x,t)\in\mathbb{R}^{+}\times\mathbb{R}^{+}|0<x/t<s_{r}\}.
\end{align*}
Hence, for any given $s_{r}>\eta_{2}>\eta_{1}>0$, we have
\begin{equation}\label{eq:lower-bound-u2}
	 \underset{t\rightarrow\infty}\liminf\,\underset{(s_{r}-\eta_{2})t\leqslant  x\leqslant (s_{r}-\eta_{1})t }\inf\,u(x,t)=\underset{\varepsilon\rightarrow0}\liminf\, \underset{s_{r}-\eta_{2}\leqslant  y\leqslant s_{r}-\eta_{1}}\inf\,u^{\varepsilon}(y,1)>0,
\end{equation}
Combining \eqref{eq:u-x0-r1} and \eqref{eq:lower-bound-u2}, we obtain ${c}_{u}^r\geqslant s_{r}$,   and the proof is complete.
\end{proof}

A similar result holds as follows.
\begin{corollary}\label{coro:speed-l}
Assume that {\rm (I$^{\lambda}$)} holds for some $\lambda\in(0,\infty]$, and let $\rho\in C((-\infty,0];[0,\infty))$  satisfy  the boundary conditions
		\begin{align}\label{eq:vis-boundary-condition-left}
			\rho(0)=0, \quad  \underset{s\rightarrow-\infty}\lim\, \rho(s)/|s|=\lambda_{1}^l\in(0,\infty].
		\end{align}
If 	$\rho$	is a viscosity subsolution of \eqref{eq:hj-tindep-up} and a viscosity supersolution of \eqref{eq:hj-tindep-low} on $\mathbb{R}^{-}$, and if there exists a unique $s_{l}<0$ such that
		\begin{align*}
			\rho(s)=0 \quad \text{for }s\in[s_{l},0],\quad \text{and} \quad  \rho(s)>0 \quad \text{for }s\in(-\infty,s_{l}),
		\end{align*}
		then $c_{u}^l=s_{l}$.
\end{corollary}

\begin{corollary}\label{coro:unique-vis}
Assume that {\rm (I$^{\lambda}$)}  holds for some $\lambda\in(0,\infty]$.
If $\rho\in C(\mathbb{R};[0,\infty))$ satisfies the assumptions of both Proposition \ref{prop:speed-r} and Corollary \ref{coro:speed-l}, then as $\varepsilon\rightarrow 0$, $w^{\varepsilon}$ defined by \eqref{eq:hopf} converges locally uniformly to $t\rho\in C(\mathbb{R}\times\mathbb{R}^{+})$.
\end{corollary}
\begin{proof}
From \eqref{eq:vis-r-sub-control}, \eqref{eq:vis-r-sup-control} and the definition of half-relaxed limit, we obtain
	\begin{align*}
		\rho\leqslant \rho_{*}\leqslant \rho^{*}\leqslant \rho \quad \text{on }\mathbb{R}^{+}.
		\end{align*}
The same inequality holds on $\mathbb{R}^{-}$. Since  $\rho_{*}(0)=\rho^{*}(0)$, it follows that
	\begin{align*}
		\rho_{*}= \rho^{*}=\rho \quad  \text{on }\mathbb{R}.
	\end{align*}
	By Lemma \ref{lem:samehj}, we deduce that
	\begin{align*}
		w_{*}=w^{*}=t\rho \quad \text{on }\mathbb{R}\times\mathbb{R}^{+}.
	\end{align*}
The result now follows from \eqref{eq:half-relaxed}.
\end{proof}

\subsection{Rightward spreading speed}\label{subsection-Rspeed}
In this subsection, we study the rightward spreading speed by constructing a series of viscosity super  and subsolutions.
We begin by establishing some auxiliary results.

Since $L_{1,+}(\cdot)>L_{1,-}(\cdot)$,  we claim that for any fixed $c_{e}\geqslant c^{*}_{1,-}$,  the equation in $p$ as follows
\begin{align}\label{eq:check-hat-p-root-right}
	c_{e}p-H_{1,+}(p)=L_{1,-}(c_{e})
\end{align}
has two real roots $\check p_{1}=\check p_{1}(c_{e})<L_{1}'(c_{e})<\hat p_{1}=\hat p_{1}(c_{e})$. Indeed, define an auxiliary function:
\begin{align*}
	F_{r}(c_{e},p)=c_{e}p-H_{1,+}(p)-L_{1,-}(c_{e}).
\end{align*}
It is straightforward to check  that $F(c_{e},p)$ is strictly increasing for $p\in(0,L_{1}'(c_{e}))$ and strictly decreasing for $p\in(L_{1}'(c_{e}),\infty)$. Note that  $F(c_{e},\infty)<0$ (due to the superlinearity of $H_{1,+}$ in $p$) and
\begin{align*}
	F_{r}(c_{e},L_{1}'(c_{e}))=L_{1,+}(c_{e})>L_{1,-}(c_{e})>0,\quad 	 F_{r}(c_{e},0)=-H_{1,+}(0)-L_{1,-}(c_{e})<0.
\end{align*}
Then the claim follows from the monotonicity of $F_{r}$ in $p$.
Particularly, we  have $\check{p}_{1}(c^{*}_{1,-})=\mu_{0,1}$, where $\mu_{0,1}$ is the smallest root of $c_{1,+}(\mu)=c^{*}_{1,-}$.
Recall that $f_{+}(\mu)= c_{e}\mu-H_{1,+}(\mu)$ is strictly increasing in $(0,L_{1}'(c_{e}))$ and strictly decreasing in $(L_{1}'( c_{e}),\infty)$.  It follows that
\begin{align}\label{eq:aux-right-func}
		f_{+}(\lambda_{1}^r)<L_{1,-}( c_{e})\quad \text{for }\lambda_{1}^r\in(0,\check{p}_{1}), \quad f_{+}(\lambda_{1}^r)>L_{1,-}(c_{e})\quad \text{for }\lambda_{1}^r\in(\check{p}_{1},L'_{1}(c_{e})].
\end{align}

\begin{lemma}\label{lem:well-p*} We have the following assertions.
	\begin{enumerate}
		\item[\rm(a)] In the region $\hat{\mathcal{R}}_{1}\cap\{\lambda_{1}^r< \check{p}_{1}\}$, where $\hat{\mathcal{R}}_{1}$ is given by \eqref{eq:region-R}, the following  equation for $p$
		\begin{align}\label{eq:root-right-1}
		c_{e} p-H_{1,-}(p)= c_{e}\lambda_{1}^r-H_{1,+}(\lambda_{1}^r)
		\end{align}
		admits the smallest root $p^{*}_{1}=p^{*}_{1}(c_{e},\lambda_{1}^r)\in(\lambda_{1}^r,L'_{1}(c_{e}))$.
		\item[\rm(b)] In the  region $\hat{\mathcal{R}}_{1}\cap\{\lambda_{1}^r=\check{p}_{1}\}$,  the equation \eqref{eq:root-right-1} admits the unique root $p^{*}_{1}=L'_{1}(c_{e})$.
	\end{enumerate}
\end{lemma}
\begin{proof}
We first prove part (a). For any given $c_{e}>c^{*}_{1,-}$, the root $\check p_{1}$ is uniquely determined, so the region $\hat{\mathcal{R}}_{1}\cap\{\lambda_{1}^r< \check{p}_{1}\}$ is nonempty. By \eqref{eq:aux-right-func}, we have
	\begin{align*}
		c_{e}\lambda_{1}^r-H_{1,+}(\lambda_{1}^r)< L_{1,-}(c_{e})=\underset{p\in \mathbb{R}}\sup\,\{c_{e} p-H_{1,-}(p)\}.
	\end{align*}
Define an auxiliary function:
	\begin{align*}
		G_{r}(c_{e},\lambda_{1}^r,p)=c_{e} p-H_{1,-}(p)-c_{e}\lambda_{1}^r+H_{1,+}(\lambda_{1}^r).
	\end{align*}
Clearly, $G_{r}$ is strictly increasing for  $p\in (0,L'_{1}(c_{e}))$ and strictly decreasing for  $p\in (L'_{1}(c_{e}),\infty)$. By \eqref{eq:aux-right-func}, we obtain
	\begin{align*}
	G_{r}(c_{e},\lambda_{1}^r,\lambda_{1}^r)=-\alpha_{-}+\alpha_{+}<0, \quad G_{r}(c_{e},\lambda_{1}^r,L'_{1}(c_{e}))=L_{1,-}(c_{e})-(c_{e}\lambda_{1}^r-H_{1,+}(\lambda_{1}^r))> 0.
	\end{align*}
	By the monotonicity of $G_{r}$ in $p$,   there exists a unique $p^{*}_{1}\in(\lambda_{1}^r,L'_{1}(c_{e}))$ such that $G_{r}(c_{e},\lambda_{1}^r,p^{*}_{1})=0$. Moreover, due to the monotonicity of $G_{r}$ and the superlinearity of $H_{1,-}$  in $p$, one can verify that $p^{*}_{1}$ is   the smallest root of $G_{r}(c_{e},\lambda_{1}^r,p)=0$.

Part (b) follows directly from the definition of $p^{*}_{1}$ since the global maximum point of the Lagrangian $L_{1,-}$ is unique. The proof is complete.
\end{proof}

\begin{lemma}\label{lem:p*-equiv}
	For any $\lambda_{1}^r$ and $c_{e}$ such that $p^{*}_{1}=p^{*}_{1}(c_{e},\lambda_{1}^r)$ is well-defined, we have
	\begin{align*}
		c_{e}>(resp. \, <\, or\, =)~k_{1}(\lambda_{1}^r)	 \Longleftrightarrow p^{*}_{1}<(resp. \, >\, or\, =)~\mu^{*}_{1,-}, \quad \forall \lambda_{1}^r\in(0,\mu^{*}_{1,-}).
	\end{align*}
\end{lemma}
\begin{proof}
	We only prove the  case ``$>$'', the other cases are similar. Note that when $c_{e}\geqslant c^{*}_{1,-}$ and $p^{*}_{1}$ is well-defined, it holds that $\mu^{*}_{1,-}\leqslant L'_{1}(c_{e})$. Since $p^{*}_{1}\in (\lambda_{1}^r,L'_{1}(c_{e})]$, and $\hat f(\mu)= c_{e}\mu-H_{1,-}(\mu)$ is increasing in $(0,L'_{1}(c_{e})]$, we have
	\begin{equation*}
	\begin{aligned}
		c_{e}>k_{1}(\lambda_{1}^r)\quad \Longleftrightarrow \quad &	 c_{e}\lambda_{1}^r-H_{1,+}(\lambda_{1}^r)<	c_{e} \mu^{*}_{1,-}-H_{1,-}(\mu^{*}_{1,-})\\
		\quad \Longleftrightarrow \quad &c_{e}p^{*}_{1}-H_{1,-}(p^{*}_{1})<c_{e} \mu^{*}_{1,-}-H_{1,-}(\mu^{*}_{1,-})\\
		 \Longleftrightarrow \quad & p^{*}_{1}<\mu^{*}_{1,-},
	\end{aligned}
	\end{equation*}
where the last step follows from the monotonicity of $\hat f$.	The proof is complete.
\end{proof}

Based on the division given in Lemma \ref{lem:right-division}, the following four lemmas address the four cases in Theorem \ref{thm:u-speed-exp-right}, which constitute the proof of this theorem.

\begin{lemma}\label{lem:right-A}
If $(\lambda_{1}^r, c_e)\in V_{1,r}^{a}\cup \gamma_{1}^{a}\cup \gamma_{1}^{b}$, then $c_{u}^r=s_{1,+}^r$.
\end{lemma}
\begin{proof}
	We consider two cases: (a) $\lambda_{1}^r\in(0,\mu^{*}_{1,+})$, and (b) $\lambda_{1}^r\in[\mu^{*}_{1,+},\infty)$.

	For case (a), we define
	\begin{equation}\label{eq:right-Aa}
		\rho(s)=\max\left\{0,\lambda_{1}^r s-H_{1,+}(\lambda_{1}^r)\right\} \quad \text{for }s\in [0,\infty).
	\end{equation}
	For case (b), we define
	\begin{equation}\label{eq:right-Ab}
		\rho(s)=
	\begin{cases}
		\lambda_{1}^r s-H_{1,+}(\lambda_{1}^r) & \text{for } s\in[H_{1}'(\lambda_{1}^r ),\infty),\\
		L_{1,+}(s) & \text{for } s\in[c^{*}_{1,+},H_{1}'(\lambda_{1}^r )),\\
		0& \text{for } s\in[0,c^{*}_{1,+}).
	\end{cases}
	\end{equation}
	We claim that $\rho$, as defined above, is a viscosity subsolution of \eqref{eq:hj-tindep-up} and a viscosity supersolution of \eqref{eq:hj-tindep-low} on $\mathbb{R}^{+}$. We only provide the details for case (b), and case (a) is similar.

	Clearly, $\rho\in C^{1}(\mathbb{R}^{+}\backslash\{c^{*}_{1,+}\})$ and satisfies \eqref{eq:vis-boundary-condition} with $s_{r}=c^{*}_{1,+}$.
Since $c_{e}\leqslant c^{*}_{1,+}$, $\rho$ is a classical solution of \eqref{eq:hj-tindep-up} on $\mathbb{R}^{+}\backslash\{c^{*}_{1,+}\}$, and $\rho(s)>0$ iff $s>c^{*}_{1,+}$. Hence, $\rho$ is a viscosity subsolution of \eqref{eq:hj-tindep-up} on $\mathbb{R}^{+}$.

Now prove that $\rho$ is also a viscosity supersolution of \eqref{eq:hj-tindep-low} on $\mathbb{R}^{+}$. By (FU) and the definitions \eqref{eq:shift-under-1}-\eqref{eq:shift-under-3}, $\rho$ is a classical solution of \eqref{eq:hj-tindep-low} on $\mathbb{R}^{+}\backslash\{c^{*}_{1,+}\}$. It remains to consider the case where  $\rho-\phi$ attains its local minimum at $c^{*}_{1,+}$  for any test function $\phi\in C^{1}(\mathbb{R}^{+})$. One can easily verify that  $\phi'(c^{*}_{1,+})\in[0,\mu^{*}_{1,+}]$. Moreover, since $\underline{R}^{*}(c^{*}_{1,+})=\underline{R}_{i}^{*}(c^{*}_{1,+})=r_{1}\alpha_{+}$ for $i=1,2,3$,  it holds at $c^{*}_{1,+}$ that
	\begin{align*}
		-c^{*}_{1,+}\phi'+H_{1}(\phi')+\underline{R}^{*}(c^{*}_{1,+})=		 -c^{*}_{1,+}\phi'+H_{1,+}(\phi')\geqslant 0,
	\end{align*}
	where the last inequality follows from the monotonicity of $c_{1,+}(\cdot)$ on $\mathbb{R}^{+}$. Hence, $\rho$ is a viscosity supersolution on $\mathbb{R}^{+}$.

This lemma follows from Proposition \ref{prop:speed-r}.
\end{proof}

\begin{lemma}\label{lem:right-B}
	If  $(\lambda_{1}^r, c_e)\in V_{1,r}^{b}\cup \gamma_{1}^{d}$, then $c_{u}^r=c_{e}$.
\end{lemma}
\begin{proof}
	Since $c_{e}>c^{*}_{1,+}$, by the monotonicity of $c_{1,+}(\cdot)$ on $\mathbb{R}^{+}$, there exist  $\nu_{1}<\mu^{*}_{1,+}<\nu_{2}$ such that $c_{e}=c_{1,+}(\nu_{i}),\,i=1,2$. In particular,
$(\nu_{1},c_{e})\in \gamma_{1}^{a}$. We consider two cases: (a) $\lambda_{1}^r\in(\nu_{1},\nu_{2})$, and (b) $\lambda_{1}^r\in[\nu_{2},\infty)$.

	\textbf{Case (a).} Define
	\begin{equation}\label{eq:right-Ba}
		\rho(s)=\min\left\{\rho_{1}(s),\rho_{2}(s)\right\} \quad \text{for }s\in [0,\infty),
	\end{equation}
where
\begin{align*}
	 \rho_{1}(s)=\max\left\{0,\lambda_{1}^rs-H_{1,+}(\lambda_{1}^r)\right\}, \quad \rho_{2}(s)=\max\left\{0,\nu_{2}s-H_{1,+}(\nu_{2})\right\}.
\end{align*}
It is easy to verify that there exists a unique $\hat s>c_{e}$ such that $\rho_{1}(\hat s)=\rho_{2}(\hat s)$, $\rho$ is of $C^{1}$ on $\mathbb{R}^{+}\backslash\{c_{e},\hat s\}$ and satisfies \eqref{eq:vis-boundary-condition} with $s_{r}=c_{e}$, and $\rho(s)>0$ iff $s>c_{e}$.
By  Proposition \ref{prop:speed-r}, it suffices to show that $\rho$ is a viscosity subsolution of \eqref{eq:hj-tindep-up} and a viscosity supersolution of \eqref{eq:hj-tindep-low} on $\mathbb{R}^{+}$.

First, since $\rho$ is a classical solution of \eqref{eq:hj-tindep-up} on $\mathbb{R}^{+}\backslash\{c_{e},\hat s\}$, hence it is an a.e. viscosity subsolution of \eqref{eq:hj-tindep-up} on $\mathbb{R}^{+}$. By \cite[Chap. II, Proposition 5.1]{bardi1997} (it can be extended to \eqref{eq:hj-tindep-up} with a minor modification), we know that $\rho$ is a viscosity subsolution of \eqref{eq:hj-tindep-up} on $\mathbb{R}^{+}$.

Next prove that $\rho$ is a viscosity supersolution of \eqref{eq:hj-tindep-low} on $\mathbb{R}^{+}$. Note that $\underline{R}$ in \eqref{eq:hj-tindep-low} is given by \eqref{eq:shift-under-1} since  $c_{e}>s_{1,+}^r> s_{2,-}^r$. Then $\rho$ is a classical solution of \eqref{eq:hj-tindep-low} on $\mathbb{R}^{+}\backslash\{c_{e},\hat s\}$. At $\hat s$, since $\rho'(\hat s^{-})=\nu_{2}>\lambda_{1}^r=\rho'(\hat s^{+})$, the viscosity supersolution definition is automatically satisfied. It remains to consider the case where $\rho-\phi$ attains its local minimum at $c_{e}$. Then $\phi'(c_{e})\in[0,\nu_{2}]$, and at $c_{e}$, we have
\begin{align*}
	-c_{e}\phi'+H_{1}(\phi')+\underline{R}^{*}(c_{e})\geqslant -c^{*}_{1,-}\phi'+H_{1,-}(\phi')\geqslant 0,
\end{align*}
where the last inequality follows from the monotonicity of $c_{1,-}(\cdot)$ on $\mathbb{R}^{+}$.

\textbf{Case (b).} Define
\begin{equation}\label{eq:right-Bb}
	\rho(s)=
\begin{cases}
	\lambda_{1}^rs-H_{1,+}(\lambda_{1}^r) & \text{for }s\in[H_{1}'(\lambda_{1}^r),\infty),\\
	L_{1,+}(s) & \text{for }s\in[H_{1}'(\nu_{2}),H_{1}'(\lambda_{1}^r)),\\
	\max\left\{0,\nu_{2}s-H_{1,+}(\nu_{2})\right\}  & \text{for }s\in[0,H_{1}'(\nu_{2})).
\end{cases}
\end{equation}
Then $\rho\in C^{1}(\mathbb{R}^{+}\backslash\{c_{e}\})$ and it satisfies \eqref{eq:vis-boundary-condition} with $s_{r}=c_{e}$, and it is a classical solution   of \eqref{eq:hj-tindep-low} on $\mathbb{R}^{+}\backslash\{c_{e}\}$ with $\underline{R}=\underline{R}_{1}$. The remaining verification that  $\rho$ is  a viscosity supersolution of \eqref{eq:hj-tindep-low} on $\mathbb{R}^{+}$ is similar to case (a). Besides, one can verify similarly as before that $\rho$ is a viscosity subsolution of \eqref{eq:hj-tindep-up} on $\mathbb{R}^{+}$.
 The conclusion follows from Proposition \ref{prop:speed-r}.
\end{proof}

\begin{lemma}\label{lem:right-C}
	If $(\lambda_{1}^r, c_e)\in V_{1,r}^{d}$, then $c_{u}^r=c_{1,-}(p^{*}_{1})$.
\end{lemma}
\begin{proof}[Sketch of proof]
By Lemmas \ref{lem:well-p*} and \ref{lem:p*-equiv}, $p^{*}_{1}$ is well-defined and $p^{*}_{1}< \mu^{*}_{1,-}$. Define
	\begin{equation}\label{eq:right-C0}
		\rho(s)=
	\begin{cases}
		\lambda_{1}^rs-H_{1,+}(\lambda_{1}^r) &\text{for }s\in[c_{e},\infty), \\
		 \max\left\{0,p^{*}_{1} s-H_{1,-}(p^{*}_{1}) \right\}& \text{for }s\in[0,c_{e}).
	\end{cases}
	\end{equation}
Then $\rho\in C^{1}(\mathbb{R}^{+}\backslash\{c_{1,-}(p^{*}),c_{e}\})$ and it satisfies \eqref{eq:vis-boundary-condition} with $s_{r}=c_{1,-}(p^{*}_{1})$. Under the assumption of this lemma, $\underline{R}$ in \eqref{eq:hj-tindep-low} is given by \eqref{eq:shift-under-1}. Using arguments similar to case (a) in the proof of Lemma \ref{lem:right-B} (noting that $p^{*}_{1}< \mu^{*}_{1,-}$ and $c_{1,-}(p^{*}_{1})>s_{2,-}^r$ in this case), we have that  $\rho$ is a viscosity subsolution of \eqref{eq:hj-tindep-up} and a viscosity supersolution of \eqref{eq:hj-tindep-low} on $\mathbb{R}^{+}$.  The result follows from Proposition \ref{prop:speed-r}.
\end{proof}

\begin{lemma}\label{lem:right-D}
	If $(\lambda_{1}^r, c_e)\in V_{1,r}^{c}\cup \gamma_{1}^{c}$, then $c_{u}^r=c^{*}_{1,-}$.
\end{lemma}
\begin{proof}
	Under the assumption of this lemma, the function $\underline{R}$ in \eqref{eq:hj-tindep-low} is given by \eqref{eq:shift-under-1}.
	For any fixed $c_{e}> c^{*}_{1,-}$, the monotonicity of $k_{1}(\mu)$ (see Lemma \ref{lem:well-division}) implies the existence and uniqueness of $\lambda_{e}\in(\mu_{0,1},\mu^{*}_{1,-})$ such that $c_{e}=k_{1}(\lambda_{e})$.
We  claim that $\lambda_{e}<\check p_{1}<\hat p_{1}$, where $\check p_{1}$ and $\hat p_{1}$ are two real roots of \eqref{eq:check-hat-p-root-right}. Indeed, by the definition of $\check{p}_{1}$, we have
	\begin{align*}
			 c_{e}\lambda_{e}-H_{1,+}(\lambda_{e})=c_{e}\mu^{*}_{1,-}-H_{1,-}(\mu^{*}_{1,-})< L_{1,-}(c_{e})=c_{e}\check{p}_{1}-H_{1,+}(\check{p}_{1}).
	\end{align*}
Using the monotonicity of $f_{+}(\mu)=c_{e}\mu-H_{1,+}(\mu)$, we conclude that $	 \lambda_{e}<\check{p}_{1}$.

We now consider three cases: (a) $\lambda_{1}^r\in[\lambda_{e},\check p_{1}]$, (b) $\lambda_{1}^r\in(\check p_{1},\hat p_{1})$, and (c) $\lambda_{1}^r\in[\hat p_{1},\infty)$.

	\textbf{Case (a).} By Lemma \ref{lem:p*-equiv} and the definition of $p^{*}_1$, we have $p^{*}_1\geqslant \mu^{*}_{1,-}$. Define
	\begin{equation}\label{eq:right-Da}
		\rho(s)=
	\begin{cases}
		\lambda_{1}^r s -H_{1,+}(\lambda_{1}^r) &\text{for~}s\in[c_{e},\infty),\\
		p^{*}_{1} s-H_{1,-}(p^{*}_{1})&\text{for~}s\in[H_{1}'(p^{*}_{1}),c_{e}),\\
		L_{1,-}(s)	 &\text{for~}s\in[c^{*}_{1,-},H_{1}'(p^{*}_{1})),\\
  0&\text{for~}s\in[0,c^{*}_{1,-}).
	\end{cases}
	\end{equation}
Then $\rho\in C^{1}(\mathbb{R}^{+}\backslash\{c^{*}_{1,-},c_{e}\})$ and it satisfies \eqref{eq:vis-boundary-condition} with $s_{r}=c_{1,-}^{*}$. The verification that $\rho$ is  a viscosity subsolution of \eqref{eq:hj-tindep-up} and a viscosity supersolution of \eqref{eq:hj-tindep-low} on $\mathbb{R}^{+}$ is similar to the proof of Lemma \ref{lem:right-C}. Proposition \ref{prop:speed-r} then implies that $c_{u}^r=c^{*}_{1,-}$.

	\textbf{Case (b).} By  the monotonicity of $f_{+}(\mu)=c_{e}\mu -H_{1,+}(\mu)$, we have
	\begin{align*}
		c_{e}\check p_{1}-H_{1,+}(\check p_{1})=	c_{e}\hat p_{1} -H_{1,+}(\hat p_{1})=L_{1,-}(c_{e})<c_{e} \lambda_{1}^r -H_{1,+}(\lambda_{1}^r).
	\end{align*}
Since $\lambda_{1}^r<\hat p$, we define
	\begin{equation}\label{eq:right-Db}
		\rho(s)=
		\begin{cases}
		\lambda_{1}^r s-H_{1,+}(\lambda_{1}^r) &\text{for~} s\in[\hat s,\infty),\\
		\hat p_{1} s-H_{1,+}(\hat p_{1}) &\text{for~} s\in[c_{e},\hat s),\\
			L_{1,-}(s)&\text{for~}s\in[c^{*}_{1,-},c_{e}),\\
	0&\text{for~}s\in[0,c^{*}_{1,-}),
		\end{cases}
	\end{equation}
	where $\hat s>c_{e}$ is the unique real number such that $\lambda_{1}^r \hat s -H_{1,+}(\lambda_{1}^r)=\hat p_{1} \hat s-H_{1,+}(\hat p_{1})$. Using Proposition \ref{prop:speed-r} and arguments similar to case (a) in the proof of Lemma \ref{lem:right-B}, one shows that $\rho$ satisfies the required viscosity condition, we omit the details here.

	\textbf{Case (c).} Define
	\begin{equation}\label{eq:right-Dc}
		\rho(s)=
		\begin{cases}
		\lambda_{1}^r s-	 H_{1,+}(\lambda_{1}^r)&\text{for~}s\in[H'_{1}(\lambda_{1}^r),\infty),\\
		L_{1,+}(s)&\text{for~}s\in[H'_{1}(\hat p_{1}),H'(\lambda_{1}^r)),\\
	\hat p_{1} s-H_{1,+}(\hat p_{1}) &\text{for~}s\in[c_{e},H'_{1}(\hat p_{1})),\\
	L_{1,-}(s)&\text{for~}s\in[c^{*}_{1,-},c_{e}),\\
	0&\text{for~}s\in[0,c^{*}_{1,-}).
		\end{cases}
	\end{equation}
The proof is similar to case (b) in Lemma \ref{lem:right-B}, and we omit details.  The proof is complete.
\end{proof}

\begin{proof}[\bf{Proof of Theorem \ref{thm:u-speed-exp-right}}]
	The result follows directly from Lemmas \ref{lem:right-A}-\ref{lem:right-D}.
\end{proof}

The following lemma deals with initial data satisfying (I$^{\infty}$).

\begin{lemma}\label{lem:right-compact}
	When {\rm (I$^{\infty}$)} holds,   we have that
	\begin{equation}\label{eq:speed-right-compact}
		s_{u}^r=
	\begin{cases}
		c^{*}_{1,+} & \text{if }c_{e}\in(-\infty,c^{*}_{1,+}],  \\
		c_{e} & \text{if }c_{e}\in(c^{*}_{1,+},c^{*}_{1,-}),  \\
		c^{*}_{1,-} & \text{if }c_{e}\in[c^{*}_{1,-},\infty).
	\end{cases}
	\end{equation}
\end{lemma}
\begin{proof}
	We only consider the case $c_{e}\in(c^{*}_{1,+},c^{*}_{1,-})$, and the other cases are similar.
Define
\begin{equation}\label{eq:right-compact}
	\rho(s)=
\begin{cases}
	L_{1,+}(s) & \text{for }s\in[H_{1}'(\nu_{2}),\infty),\\
	\max\left\{0,\nu_{2}s-H_{1,+}(\nu_{2})\right\}  & \text{for }s\in[0,H_{1}'(\nu_{2})).
\end{cases}
\end{equation}
Note that  as $\lambda_{1}^r\rightarrow\infty$,  \eqref{eq:right-Bb} converges to \eqref{eq:right-compact} in $C_{loc}(\mathbb{R}^{+})$. By the stability property of the viscosity solution as in \cite[Theorem 6.2]{barles2013},  $\rho$ defined by \eqref{eq:right-compact} is a viscosity subsolution of \eqref{eq:hj-tindep-up}, and a viscosity supersolution of \eqref{eq:hj-tindep-low} on $\mathbb{R}^{+}$ with $\underline{R}$ given by \eqref{eq:shift-under-1}. By the superlinearity of $L_{1,+}$, we obtain \eqref{eq:vis-boundary-condition}. The conclusion follows from Proposition \ref{prop:speed-r}.
\end{proof}

\subsection{Leftward spreading speed}\label{subsection-left-u}
In this subsection, we study the leftward spreading speed by constructing a series of viscosity super- and subsolutions.
We first establish some auxiliary results. For notational simplicity, we sometimes write $\tilde c_{e}=-c_{e}$.

First, for the case $c_{e}\leqslant -s_{1,-}^l$, we define auxiliary functions
\begin{center}
	$F_{l}(\tilde c_{e},p)=\tilde c_{e}p-H_{1,+}(p)-L_{1,-}(\tilde c_{e})$,\\
	$G_{l}(\tilde c_{e},\lambda_{1}^l,p)=\tilde c_{e} p-H_{1,+}(p)-\tilde c_{e}\lambda_{1}^l+H_{1,-}(\lambda_{1}^l)$,
\end{center}
and sets
\begin{center}
	$\overline{\mathcal{E}}_{1}:=\{\tilde c_{e}|\tilde c_{e}\geqslant c^{*}_{1,-}=H'_{1}(\mu^{*}_{1,-})\},$\\
	$\underline{\mathcal{E}}_{1}:=\{(\tilde c_{e},\lambda^{l}_{1})|\tilde c_{e}\geqslant c_{1,-}(\lambda^{l}_{1}) \text{~for~}\lambda^{l}_{1}\leqslant \mu^{*}_{1,-}\text{~or~}\tilde c_{e}\geqslant H_{1}'(\lambda^{l}_{1}) \text{~for~}\lambda^{l}_{1}>\mu^{*}_{1,-}\}$.
\end{center}

\begin{lemma}\label{lem:left-para-p}
	The following assertions hold.
	\begin{enumerate}
		\item[\rm(a)] In $\underline{\mathcal{E}}_{1}$, $G_{l}(\tilde c_{e},\lambda_{1}^l,p)=0$ admits the smallest root $\underline{p}_{1}=\underline{p}_{1}(\tilde c_{e},\lambda_{1}^l)\in(0,\lambda_{1}^l\land L'_{1}(\tilde c_{e}))$.
		\item[\rm(b)] In $\overline{\mathcal{E}}_{1}$, $F_{l}(\tilde c_{e},p)=0$ admits the smallest root $\bar{p}_{1}=\bar{p}_{1}(\tilde c_{e})\in(0,L'_{1}(\tilde c_{e}))$. Moreover, $\bar{p}_{1}$ is strictly increasing in $\tilde c_{e}$, and there exists a unique $\bar{c}_{1}>c^{*}_{1,-}$ such that $\bar{p}_{1}(\bar{c}_{1})=\mu^{*}_{1,+}$.
	\end{enumerate}
\end{lemma}
The proof of this lemma is a modification of that of Lemma \ref{lem:well-p*}, we omit the details and refer to \cite{lam2022a,tao-shift}. By the monotonicity of $\bar{p}_{1}$ in $\tilde c_{e}$, we have the following result, whose proof is similar to that of Lemma \ref{lem:p*-equiv}.

\begin{lemma}\label{lem:left-pp-equiv}
 When $\bar{p}_{1}$ or $\underline{p}_{1}$ is well-defined, the following equivalences hold.
	\begin{enumerate}
		\item[\rm(a)] $\bar{p}_{1}(\tilde c_{e})< (resp.\,>\,or\,=)~ \mu^{*}_{1,+} \Longleftrightarrow \tilde c_{e}<(resp.\,>\,or\,=)~\bar{c}_{1}$.
		\item[\rm(b)] $\underline{p}_{1}(\tilde c_{e},\lambda_{1}^l)< (resp.\,>\,or\,=)~\mu^{*}_{1,+} \Longleftrightarrow \tilde c_{e}<(resp.\,>\,or\,=)~g_{1}(\lambda_{1}^l)$, $\forall \lambda_{1}^l\in(\mu^{*}_{1,+},L'_{1}(\bar{c}_{1})]$.
	\end{enumerate}
\end{lemma}

Based on  the division in Lemma \ref{lem:left-division}, the following four lemmas address the four cases in Theorem \ref{thm:u-speed-exp-left}, which constitute the proof of this theorem.
\begin{lemma}\label{lem:left-A}
	If $(\lambda_{1}^l,c_{e})\in V_{1,l}^{d}\cup \gamma_{1}^{q}\cup \gamma_{1}^{r}$, then $c_{u}^l=-c_{1,+}^{*}$.
\end{lemma}

\begin{proof}
Consider two cases:
\begin{itemize}
\item [(a)] $\lambda_{1}^l\in[\mu^{*}_{1,+},L'_{1}(\bar{c}_{1})]$ and $c_{e}\leqslant -g_{1}(\lambda_{1}^l)$, or $\lambda_{1}^l\in (L'_{1}(\bar{c}_{1}),\infty)$ and $c_{e}<-H_{1}(\lambda_{1}^l)$;
\item [(b)] $\lambda_{1}^l\in (L'_{1}(\bar{c}_{1}),\infty)$ and $-H_{1}(\lambda_{1}^l)\leqslant c_{e}\leqslant -\bar{c}_{1}$.
\end{itemize}

\textbf{Case (a).} By Lemmas \ref{lem:left-para-p}-(a) and \ref{lem:left-pp-equiv}-(b), we have $\underline{p}_{1} \geqslant \mu^{*}_{1,+}$. Define
\begin{equation}\label{eq:left-Aa}
	\rho(s)=
\begin{cases}
	-\lambda_{1}^l s-H_{1,-}(\lambda_{1}^l) & \text{for }s\in(-\infty,c_{e}], \\
	-\underline{p}_{1} s-H_{1,+}(\underline{p}_{1}) & \text{for }s\in(c_{e},-H_{1}'(\underline{p}_{1})], \\
	L_{1,+}(s) & \text{for }s\in(-H_{1}'(\underline{p}_{1}),-c^{*}_{1,+}], \\
	0& \text{for }s\in(-c^{*}_{1,+},0].
\end{cases}
\end{equation}
Then $\rho\in C^{1}(\mathbb{R}^{-}\backslash\{c_{e},-c^{*}_{1,+}\})$, which satisfies \eqref{eq:vis-boundary-condition-left} with $s_{l}=-c^{*}_{1,+}$, and it is a classical solution of \eqref{eq:hj-tindep-up} for $s\not\in\{c_{e},-c^{*}_{1,+}\}$. Hence,   $\rho$ is a viscosity subsolution of \eqref{eq:hj-tindep-up} on $\mathbb{R}^{-}$.

To show that $\rho$ is a viscosity supersolution of \eqref{eq:hj-tindep-low} on $\mathbb{R}^{-}$, note that $\underline{R}$ is given by \eqref{eq:shift-under-3} and $\rho$ is a classical solution of \eqref{eq:hj-tindep-low} for $s\not\in\{c_{e},-c^{*}_{1,+}\}$.
Let  $\phi\in C^{1}(\mathbb{R}^{-})$ be a test function. Next consider two cases.
If $\rho-\phi$ attains its local minimum at $c_{e}$, then $\phi'(c_{e})\in[-\lambda_{1}^l,-\underline{p}_{1}]$. Since $\underline{R}^{*}(c_{e})=r_{1}\alpha_{-}$, it holds at $c_{e}$ that
\begin{align*}
	\rho(c_{e})-c_{e}\phi'+H_{1}(\phi')+\underline{R}^{*}(c_{e})=\hat f(\lambda_{1}^l)-\hat f(-\phi')\geqslant 0,
\end{align*}
where the last inequality follows from the fact that $\hat f(\mu)=\tilde c_{e} \mu-H_{1,-}(\mu)$ is increasing for $\mu\in(0,L'_{1}(\tilde c_{e})]$.
If $\rho-\phi$ attains its local minimum at $-c^{*}_{1,+}$,  then $\phi'(c_{e})\in[-\mu^{*}_{1,+},0]$ and $\underline{R}^{*}(-c^{*}_{1,+})=r_{1}\alpha_{+}$. Thus at  $-c^{*}_{1,+}$:
\begin{align*}
	-c^{*}_{1,+}\phi'+H_{1}(\phi')+\underline{R}^{*}(-c^{*}_{1,+})=	 -c^{*}_{1,+}\phi'+H_{1,+}(\phi')\geqslant 0,
\end{align*}
by the symmetry of $H_{1,+}(\cdot)$ on $\mathbb{R}$ and the monotonicity of $c_{1,+}(\cdot)$ on $\mathbb{R}^{-}$.

\textbf{Case (b).} By Lemmas \ref{lem:left-para-p}-(b) and \ref{lem:left-pp-equiv}-(a), we have $\bar{p}_{1} \geqslant \mu^{*}_{1,+}$. Define
\begin{equation}\label{eq:left-Ab}
	\rho(s)=
\begin{cases}
	-\lambda_{1}^l s-H_{1,-}(\lambda_{1}^l) & \text{for }s\in(-\infty,-H_{1}'(\lambda_{1}^l)], \\
	L_{1,-}(s) & \text{for }s\in(-H_{1}'(\lambda_{1}^l),c_{e}], \\
	-\bar{p}_{1} s-H_{1,+}(\bar{p}_{1})  & \text{for }s\in(c_{e},-H_{1}'(\bar{p}_{1})], \\
	L_{1,+}(s) & \text{for }s\in(-H_{1}'(\bar{p}_{1}),-c^{*}_{1,+}], \\
	0& \text{for }s\in(-c^{*}_{1,+},0].
\end{cases}
\end{equation}
Then $\rho\in C^{1}(\mathbb{R}^{-}\backslash\{c_{e},-c^{*}_{1,+}\})$ satisfies \eqref{eq:vis-boundary-condition-left} with $s_{l}=-c^{*}_{1,+}$, and it is a  viscosity subsolution of \eqref{eq:hj-tindep-up} on $\mathbb{R}^{-}$ . To show that $\rho$ is a viscosity supersolution of \eqref{eq:hj-tindep-low} (with $\underline{R}$ given by \eqref{eq:shift-under-3}) on $\mathbb{R}^{-}$. We need to consider a point $s\in\{c_{e},-c^{*}_{1,+}\}$, we only verify that the point $s=c_{e}$ here, and the other one is similar. For any test function $\phi\in C^{1}(\mathbb{R}^{-})$ such that $\rho-\phi$ attains its local minimum at $c_{e}$, we have $\phi'(c_{e})\in[-L'_{1}(c_{e}),-\bar{p}_{1}]$. By the symmetry of $L_{1,-}$ on $\mathbb{R}$, we have
\begin{align*}
	 \rho(c_{e})-c_{e}\phi'+H_{1}(\phi')+\underline{R}^{*}(c_{e})=L_{1,-}(c_{e})-(c_{e}\phi'-H_{1,-}(\phi'))\geqslant 0.
\end{align*}

In both cases,  Corollary \ref{coro:speed-l} implies that $c_{u}^l=-c_{1,+}^{*}$.
\end{proof}

The following three lemmas address the remaining cases in Theorem \ref{thm:u-speed-exp-left}. We only provide the construction of viscosity solutions, as the proofs are similar to those in Lemma \ref{lem:left-A}.

\begin{lemma}\label{lem:left-B}
	If $(\lambda_{1}^l,c_{e})\in V_{1,l}^{a}\cup \gamma_{1}^{o}\cup \gamma_{1}^{p}$, then $c_{u}^l=-s_{1,-}^l$.
	\end{lemma}

	\begin{proof}[Sketch of proof]
		Consider two cases: (a) $\lambda_{1}^l\in(0,\mu^{*}_{1,-})$, and (b) $\lambda_{1}^l\in[\mu^{*}_{1,-},\infty)$.

		\textbf{Case (a).} Define
		\begin{align}\label{eq:left-Ba}
			 \rho(s)=\max\left\{0,-\lambda_{1}^ls-H_{1,-}(\lambda_{1}^l)\right\} \quad \text{for }s\in(-\infty,0].
		\end{align}

		\textbf{Case (b).} Define
		\begin{equation}\label{eq:left-Bb}
			\rho(s)=
		\begin{cases}
			-\lambda_{1}^ls-H_{1,-}(\lambda_{1}^l) & \text{for }s\in(-\infty,-H_{1}'(\lambda_{1}^l)],\\
			 L_{1,-}(s) & \text{for }s\in(-H_{1}'(\lambda_{1}^l),-c^{*}_{1,-}],\\
				0&\text{for }s\in(-c^{*}_{1,-},0].
		\end{cases}
		\end{equation}
By Corollary \ref{coro:speed-l}, we obtain   $c_{u}^l=-s_{1,-}^{l}$.
	\end{proof}

\begin{lemma}\label{lem:left-C}
	If $(\lambda_{1}^l,c_{e})\in V_{1,l}^{c}$, then $c_{u}^l=-c_{1,+}(\underline{p}_{1})$.
\end{lemma}

\begin{proof}[Sketch of proof]
By  Lemmas \ref{lem:left-para-p}-(a) and \ref{lem:left-pp-equiv}-(b), we have $\underline{p}_{1} <\mu^{*}_{1,+}$. Define
\begin{equation}\label{eq:left-C0}
	\rho(s)=
\begin{cases}
	-\lambda_{1}^l s-H_{1,-}(\lambda_{1}^l) & \text{for }s\in(-\infty,c_{e}], \\
	 \max\left\{0,-\underline{p}_{1} s-H_{1,+}(\underline{p}_{1})\right\}& \text{for }s\in(c_{e},0].
\end{cases}
\end{equation}
Corollary \ref{coro:speed-l} implies $c_{u}^l=-c_{1,+}(\underline{p}_{1})$.
\end{proof}

\begin{lemma}\label{lem:left-D}
	If $(\lambda_{1}^l,c_{e})\in V_{1,l}^{b}\cup \gamma_{1}^{s}$, then $c_{u}^l=-c_{1,+}(\bar{p}_{1})$.
\end{lemma}

\begin{proof}[Sketch of proof]
By  Lemmas \ref{lem:left-para-p}-(b) and \ref{lem:left-pp-equiv}-(a), we have $\bar{p}_{1} <\mu^{*}_{1,+}$. Define
	\begin{equation}\label{eq:left-D0}
		\rho(s)=
	\begin{cases}
		-\lambda_{1}^l s-H_{1,-}(\lambda_{1}^l) & \text{for }s\in(-\infty,-H_{1}'(\lambda_{1}^l)], \\
		L_{1,-}(s) & \text{for }s\in(-H_{1}'(\lambda_{1}^l),c_{e}], \\
			\max\left\{0,-\bar{p}_{1} s-H_{1,+}(\bar{p}_{1})\right\}& \text{for }s\in(c_{e},0].
	\end{cases}
	\end{equation}
Corollary \ref{coro:speed-l} implies $c_{u}^l=-c_{1,+}(\bar{p}_{1})$.
\end{proof}

\begin{proof}[\bf{Proof of Theorem \ref{thm:u-speed-exp-left}}]
It follows directly from Lemmas \ref{lem:left-A}-\ref{lem:left-D}.
\end{proof}

Using an argument similar to the proof of Lemma \ref{lem:right-compact}, we have the following lemma.
\begin{lemma}\label{lem:left-compact}
When {\rm (I$^{\infty}$)} holds,   we have that
	\begin{equation}\label{eq:speed-left-compact}
		s_{u}^l=
	\begin{cases}
		-c^{*}_{1,+} & \text{if }c_{e}\in(-\infty,-\bar{c}_{1}],  \\
		-c_{1,+}(\bar{p}_{1}) & \text{if }c_{e}\in(-\bar{c}_{1},-c^{*}_{1,-}),  \\
		-c^{*}_{1,-} & \text{if }c_{e}\in[-c^{*}_{1,-},\infty).
	\end{cases}
	\end{equation}
\end{lemma}

\begin{proof}[\bf{Proof of Theorem \ref{thm:u-speed-comp}}]
It follows from Lemmas \ref{lem:right-compact} and \ref{lem:left-compact}.
\end{proof}

\section{The upper bound for the spreading speed of predators}\label{sec:formula-max-predator}

In this section, we always assume that conditions (J), (A)  and (H1) hold. Using stability properties of viscosity solutions, we derive a lower limiting Hamilton-Jacobi equation, and then construct a series of viscosity subsolutions to investigate explicit formulas for the upper bound of predators' spreading speed.

For any $\delta>0$, consider an auxiliary problem:
\begin{equation}\label{eq:forced-aux}
	\psi_{t}=d_{1}(J_{1}\ast \psi-\psi)+r_{1}\psi[\alpha(x-c_{e}t)+\delta-\psi], \quad (x,t)\in\mathbb{R}\times\mathbb{R}^{+}.
\end{equation}
By \cite[Theorem 2.2]{qiao2022a}, equation \eqref{eq:forced-aux} admits a nonincreasing forced wave $\psi_{\delta}(x-c_{e}t)$ satisfying
\begin{align}\label{eq:limit-forced-wave}
	\psi_{\delta}(-\infty)=\alpha_{-}+\delta, \quad \psi_{\delta}(+\infty)=\alpha_{+}+\delta.
\end{align}
Since $u_{0}\leqslant \alpha_{-}$ on $\mathbb{R}$ and $u_{0}(x)\rightarrow 0$ as $x\rightarrow \pm \infty$, for any fixed $\delta>0$, there exists a sufficiently large constant $x_{\delta}>0$ such that $u_{0}(x)\leqslant \psi_{\delta}(x-x_{\delta})$ for $x\in \mathbb{R}$.
 Define
\begin{align}\label{eq:var-p}
	\varphi_{\delta}(x,t):=\psi_{\delta}(x-c_{e}t-x_{\delta}), \quad (x,t)\in \mathbb{R}\times[0,\infty).
\end{align}
It is straightforward to verify that $\varphi_{\delta}$ is a supersolution for the prey equation. By the comparison principle,  we obtain
\begin{align}\label{eq:u-com-forced}
	u(x,t)\leqslant \varphi_{\delta}(x,t), \quad (x,t)\in \mathbb{R}\times\mathbb{R}^{+}.
\end{align}
Now let $n_{\delta}$ be the solution of the following Cauchy problem
\begin{equation}\label{eq:v-sup}
\begin{cases}
	n_{t}=d_{2}(J_{2}\ast n-n)+r_{2}n\,[b\varphi_{\delta}-1-n], &(x,t)\in\mathbb{R}\times\mathbb{R}^{+},\\
	n(x,0)=v_{0}(x), & x\in \mathbb{R}.
\end{cases}
\end{equation}
Applying the comparison principle again yields
\begin{align}\label{eq:comparison-n-v}
	v(x,t)\leqslant n_{\delta}(x,t), \quad (x,t)\in \mathbb{R}\times\mathbb{R}^{+}.
\end{align}

\subsection{A lower limiting Hamilton-Jacobi equation}\label{subsec:v-low-hj}
For any $\varepsilon>0$,  define the hyperbolic scaling and Hopf-Cole transform  as follows:
\begin{align*}
	n_{\delta}^{\varepsilon}(x,t)=n_{\delta} \left(\frac{x}{\varepsilon},\frac{t}{\varepsilon}\right), \quad I^{{\delta},\varepsilon}(x,t)=-\varepsilon \ln n_{\delta}^{\varepsilon}(x,t), \quad (x,t)\in \mathbb{R}\times\mathbb{R}^{+}.
\end{align*}

\begin{proposition}\label{prop:local-bound-Ie}
	Assume that {\rm (I$^{\lambda}$)} holds for some $\lambda\in(0,\infty]$. Then $I^{\delta,\varepsilon}$ is locally uniformly bounded with respect to $\varepsilon$ on $\mathbb{R}\times\mathbb{R}^{+}$. In particular, its lower half-relaxed limit
	\begin{align*}
		I_*^{\delta}(x,t):=\liminf_{\substack{\varepsilon \rightarrow 0\\(y,s)\rightarrow(x,t)}}I^{\delta,\varepsilon}(y,s), \quad (x,t)\in \mathbb{R}\times\mathbb{R}^{+}
	\end{align*}
	is well-defined and constitutes a viscosity supersolution of the Hamilton-Jacobi equation
\begin{equation}\label{eq:HJ-I}
	 \min\left\{\partial_{t}I+H_{2}(\partial_{x}I)+R_{\delta}(x/t),I\right\}=0, \quad (x,t)\in \mathbb{R}\times\mathbb{R}^{+},
\end{equation}
where $H_{2}$ is given by \eqref{eq:hi}, and $R_{\delta}$ is defined by
\begin{equation}\label{eq:RD-usc}
	R_{\delta}(s)=
\left\{\begin{array}{ll}
	r_{2}[b(\alpha_{-}+\delta)-1] \quad \text{for }s\leqslant c_{e}, \\
	r_{2}[b(\alpha_{+}+\delta)-1] \quad \text{for }s> c_{e}.
\end{array}\right.
\end{equation}
Moreover, the following initial and boundary conditions hold in the classical sense:
\begin{itemize}
\item $I_{*}^{\delta}(0,t)=0$ for $t\in\mathbb{R}^{+}$,
\item if {\rm (I$^{\lambda}$)} holds for $\lambda<\infty$, then
\begin{align*}
	 I_{*}^{\delta}(x,0)=\left\{\begin{array}{ll}
		\lambda^{r}_{2}x &  \text{for~} x\in\mathbb{R}^{+}\cup\{0\},\\
	-\lambda^{l}_{2}x &  \text{for~}x\in\mathbb{R}^{-},
	\end{array}\right. 
\end{align*}
\item if {\rm (I$^{\infty}$)} holds, then
\begin{align*}
  I_{*}^{\delta}(x,0)=+\infty \quad \text{for all~}x\neq 0. 
\end{align*}
\end{itemize}
\end{proposition}
\begin{proof}[Sketch of proof]
	Since $b\varphi_{\delta}-1>\mathcal{V}_{+}>0$ on $\mathbb{R}\times\mathbb{R}^{+}$, the proof follows arguments similar to those in Propositions \ref{prop:local-bound-u} and \ref{prop:lower-hj-control}, and therefore, we omit the details.
\end{proof}
\begin{remark}\label{remark:hj-unif-bound-d}
	\rm{
By the comparison principle, we can show that the locally uniform (in $\varepsilon$) bound of $I^{\delta,\varepsilon}$ is nonincreasing  with respect to $\delta$. Note that the proof of Proposition \ref{prop:local-bound-u} also implies that $I^{0,\varepsilon}$ remains locally uniformly bounded in $\varepsilon$ on $\mathbb{R}\times\mathbb{R}^{+}$. Therefore, the following lower half-relaxed limit
\begin{align}\label{eq:lsc-I}
	I_{*}^{0}(x,t):=\liminf_{\substack{\delta \rightarrow 0\\(y,s)\rightarrow(x,t)}}I^{\delta}_{*}(y,s), \quad (x,t)\in \mathbb{R}\times\mathbb{R}^{+}
\end{align}
is well-defined.
Moreover, since $I_{*}^\delta$ is a viscosity supersolution of \eqref{eq:HJ-I}, the stability property of viscosity solutions\cite[Theorem 6.2]{barles2013} implies that $I_{*}^{0}$ defined by \eqref{eq:lsc-I} is a viscosity supersolution of the  Hamilton-Jacobi equation
\begin{align}\label{eq:hj-I0}
	\min\left\{\partial_{t}I+H_{2}(\partial_{x}I)+R_{0}(x/t),I\right\}=0, \quad (x,t)\in \mathbb{R}\times\mathbb{R}^{+},
\end{align}
where
\begin{equation}\label{eq:R0-usc}
	R_{0}(s)=
	\begin{cases}
		r_{2}\mathcal{V} _{-} &\text{for }s\leqslant c_{e}, \\
	r_{2}\mathcal{V} _{+}& \text{for }s> c_{e}.
	\end{cases}
\end{equation}
In particular, the following initial and boundary conditions hold in the classical sense:
\begin{itemize}
\item $I_{*}^{0}(0,t)\geqslant 0$ for all $t\in\mathbb{R}^{+}$,
\item if {\rm (I$^{\lambda}$)} holds for $\lambda<\infty$, then
\begin{align*}
 I_{*}^{0}(x,0)\geqslant \left\{\begin{array}{ll}
		\lambda^{r}_{2}x &  \text{for~} x\in\mathbb{R}^{+}\cup\{0\},\\
	-\lambda^{l}_{2}x &  \text{for~}x\in\mathbb{R}^{-},
	\end{array}\right.
\end{align*}
\item if {\rm (I$^{\infty}$)} holds, then
\begin{align*}
	 I_{*}^{0}(x,0)=+\infty \quad \text{for all~}x\neq 0.
\end{align*}
\end{itemize}
	}
\end{remark}

\begin{corollary}\label{coro:I-low-hj-control}
	Assume that {\rm (I$^{\lambda}$)} holds for some $\lambda\in(0,\infty]$. Then the function $\rho_{*}^{0}(x/t):=I_{*}^{0}\left({x}/{t},1\right)$ is a viscosity supersolution of the Hamilton-Jacobi equation
	\begin{equation}\label{eq:I-hj-tindep-low}
		\min\left\{\rho-s\rho'+H_{2}(\rho')+R_{0}(s),\rho\right\}=0, \quad  s\in \mathbb{R}^{+}~(or\, \mathbb{R}^{-}),
	\end{equation}
where $H_{2}$ and $R_{0}$ are given by \eqref{eq:hi} and \eqref{eq:R0-usc}, respectively. In particular, the following inequalities hold in the classical sense:
	\begin{align}\label{eq:hj-tindep-low-boundary-I}
		\rho_{*}^{0}(0)\geqslant 0, \quad \text{and} \quad \underset{s\rightarrow+\infty}\lim\, \rho_{*}^{0}(s)/s\geqslant \lambda_{2}^r~(or\, \underset{s\rightarrow-\infty}\lim\, \rho_{*}^{0}(s)/|s|\geqslant \lambda_{2}^l).
	\end{align}
\end{corollary}

The proof of this corollary follows similarly to that  of  Lemma \ref{lem:low-hj-control}, and we therefore omit it.

\subsection{Determinacy of the upper bound for the spreading speed}
\begin{proposition}\label{prop:max-speed-v}
	Assume that {\rm (I$^{\lambda}$)} holds for some $\lambda\in(0,\infty]$. We have the  following assertions.
	\begin{enumerate}
		\item[\rm(a)] 	Let  $\rho\in C([0,\infty);[0,\infty))$ satisfy  the following boundary conditions
		\begin{align}\label{eq:vis-boundary-condition-Ir}
			\rho(0)=0, \quad  \underset{s\rightarrow+\infty}\lim\, \rho(s)/s\leqslant \lambda_{2}^r\in(0,\infty].
		\end{align}
If $\rho$ is a viscosity subsolution of \eqref{eq:I-hj-tindep-low} on $\mathbb{R}^{+}$  and there exists a unique $s_{r}>0$ such that
		\begin{align}\label{eq:r>0}
			\rho(s)=0 \quad \text{for }s\in[0,s_{r}],\quad \text{and} \quad  \rho(s)>0 \quad \text{for }s\in(s_{r},\infty),
		\end{align}
		then $c_{v}^r\leqslant s_{r}$.
		\item[\rm(b)] Let $\rho\in C((-\infty,0];[0,\infty))$  satisfy  the following boundary conditions
		\begin{align}\label{eq:vis-boundary-condition-Il}
			\rho(0)=0, \quad  \underset{s\rightarrow+\infty}\lim\, \rho(s)/|s|\leqslant \lambda_{2}^l\in(0,\infty].
		\end{align}
If $\rho$ is a viscosity subsolution of \eqref{eq:I-hj-tindep-low} on $\mathbb{R}^{-}$  and there exists a unique $s_{l}<0$ such that
		\begin{align*}
			\rho(s)=0 \quad \text{for }s\in[s_{l},0],\quad \text{and} \quad  \rho(s)>0 \quad \text{for }s\in(-\infty,s_{l}),
		\end{align*}
		then $c_{v}^l\geqslant s_{l}$.
	\end{enumerate}
\end{proposition}

\begin{proof}[Sketch of proof]
We only give  a sketch of the proof of assertion (a), since (b) is similar.

	By Lemma \ref{lem:samehj} and a slight modification of the proof of Proposition \ref{prop:speed-r}, we have
	\begin{align*}
		 \underset{\varepsilon\rightarrow0}\limsup\,v^{\varepsilon}\leqslant \underset{\varepsilon\rightarrow0}\limsup\,n_{\delta}^{\varepsilon}=0 \quad  \text{locally uniformly in }\{(x,t)\in\mathbb{R}^{+}\times\mathbb{R}^{+}|\rho_{*}^\delta(x/t)>0\},
	\end{align*}
	where $v^{\varepsilon}$ is the hyperbolic scaling of $v$.
Since $\delta>0$ is arbitrary and $v$ is nonnegative, taking $\delta\rightarrow 0^{+}$, we obtain
\begin{align}\label{eq:ve-to-0}
	\underset{\varepsilon\rightarrow0}\limsup\,v^{\varepsilon}=0 \quad  \text{locally uniformly in }\{(x,t)\in\mathbb{R}^{+}\times\mathbb{R}^{+}|\rho_{*}^0(x/t)>0\},
\end{align}
where $\rho_{*}^0$ is a viscosity supersolution of \eqref{eq:I-hj-tindep-low}. Note that $\rho_{*}^{0}$ satisfies the boundary conditions \eqref{eq:hj-tindep-low-boundary-I}, and $\rho$  satisfies  \eqref{eq:vis-boundary-condition-Ir}, which is a viscosity subsolution of \eqref{eq:I-hj-tindep-low}. By the comparison principle (Lemma \ref{lem:viscosity-comparison-t-indep}) for \eqref{eq:I-hj-tindep-low}, it holds that $\rho \leqslant \rho_{*}^{0}$ on $\mathbb{R}^{+}$. Hence, \eqref{eq:r>0} implies
\begin{align}\label{eq:subset-rh-0}
	 \{(x,t)\in\mathbb{R}^{+}\times\mathbb{R}^{+}|\rho_{*}^0(x/t)>0\}\supset  \{(x,t)\in\mathbb{R}^{+}\times\mathbb{R}^{+}|x/t>s_{r}\}.
\end{align}
Combining \eqref{eq:ve-to-0} and \eqref{eq:subset-rh-0},  for any given $\eta_{2}>\eta_{1}>0$, we have
\begin{align}\label{eq:v-speed-dete-1}
		 \underset{t\rightarrow\infty}\lim\,\underset{(s_{r}+\eta_{1})t\leqslant x\leqslant (s_{r}+\eta_{2})t }\sup\,v(x,t)=\underset{\varepsilon\rightarrow0}\lim\, \underset{s_{r}+\eta_{1}\leqslant y\leqslant s_{r}+\eta_{2}}\sup\,v^{\varepsilon}(y,1)=0.
\end{align}
Meanwhile, the comparison principle also implies that
\begin{align}\label{eq:v-speed-dete-2}
	\underset{t\rightarrow\infty}\lim\,\underset{x\geqslant (s_{2,-}^{r}+\eta)t}\sup\,v(x,t)=0 \quad \text{for all }\eta>0,
\end{align}
where $s_{2,-}^{r}$ is given by \eqref{eq:speed-homo}. Thus, \eqref{eq:v-speed-dete-1} together with \eqref{eq:v-speed-dete-2} yields
\begin{align*}
	\underset{t\rightarrow\infty}\lim\,\underset{ x\geqslant (s_{r}+\eta)t }\sup\,v(x,t)=0 \quad \text{for all }\eta>0.
\end{align*}
The proof is complete.
\end{proof}

\subsection{Upper bound of the rightward spreading speed}\label{subsec:v-max-right}
In this subsection, we study the upper bound for the rightward spreading speed of predators by constructing  viscosity subsolutions for \eqref{eq:I-hj-tindep-low} and applying Proposition \ref{prop:max-speed-v}.  We only  provide the form of viscosity subsolutions and omit the rigorous verifications, which are similar to those in Subsection \ref{subsection-Rspeed}.

By the same argument as in Subsection \ref{subsection-Rspeed},  for any $ c_{e}\geqslant c^{*}_{2,-}$, the equation in $p$ as follows
\begin{align}\label{eq:check-hat-p-root-right-vequ}
	c_{e}p-H_{2,+}(p)=L_{2,-}(c_{e})
\end{align}
must have two real roots $\check p_{2}<L'_{2}(c_{e})<\hat p_{2}$. In particular, in the region $\hat{\mathcal{R}}_{2}\cap\{\lambda_{2}^r< \check{p}_{2}\}$, where $\hat{\mathcal{R}}_{2}$ is given by \eqref{eq:region-R}, the equation
\begin{align}\label{eq:root-right-1-vequ}
c_{e} p-H_{2,-}(p)= c_{e}\lambda_{2}^r-H_{2,+}(\lambda_{2}^r)
\end{align}
admits the smallest root $p^{*}_{2}=p^{*}_{2}(c_{e},\lambda_{2}^r)\in(\lambda_{2}^r,L'_{2}(c_{e}))$. Besides, in the region $\hat{\mathcal{R}}_{2}\cap\{\lambda_{2}^r=\check{p}_{2}\}$,  the equation \eqref{eq:root-right-1-vequ} admits a unique root $p^{*}_{2}=L'_{2}(c_{e})$.

\begin{lemma}\label{lem:v-right-A}
	If $\lambda_{2}^r\in(0,\infty]$ and $c_{e}\leqslant s_{2,+}^r$, then $c_{v}^r\leqslant  s_{2,+}^r$.
\end{lemma}

\begin{proof}[Sketch of proof] Set $\mu=\lambda_{2}^r\land \mu^{*}_{2,+}$ and  define
	\begin{equation}\label{eq:v-right-A0}
		\rho(s)=\max\left\{0,\mu s-H_{2,+}(\mu)\right\} \quad \text{for }s\in[0,\infty).
		\end{equation}
Then   $c_{v}^r\leqslant s_{2,+}(\mu)= s_{2,+}^r$.
\end{proof}

\begin{lemma}\label{lem:v-right-B}
	If $\lambda_{2}^r\in[\mu_{0,2},\infty]$ and $s_{2,+}^{r}<c_{e}<c^{*}_{2,-}$, then $c_{v}^r\leqslant  c_{e}$.
\end{lemma}

\begin{proof}[Sketch of proof] For any $c_{e}$ satisfying  the assumption in this lemma, the monotonicity of $c_{2,+}(\cdot)$ on $\mathbb{R}^{+}$ implies that $c_{2,+}(\mu)=c_{e}$ admits two real roots $\mu_{e}<\mu^{*}_{2,+}<\nu_{e}$. Define
	\begin{align*}
		\rho(s)=\max\left\{0,\mu_{e} s-H_{2,+}(\mu_{e})\right\} \quad \text{for }s\in[0,\infty).
	\end{align*}
Then  ${c}_{v}^r\leqslant s_{2,+}(\mu_{e})= c_{e}$.
\end{proof}

\begin{lemma}\label{lem:v-right-C}
	If either $\lambda_{2}^r\in[\mu_{0,2},\mu^{*}_{2,-})$ and $c^{*}_{2,-}\leqslant c_{e}\leqslant k_{2}(\lambda_{2}^r)$, or $\lambda_{2}^r\in[\mu^{*}_{2,-},\infty]$ and $c_{e}\geqslant c^{*}_{2,-}$,	then $c_{v}^r\leqslant c^{*}_{2,-}$.
\end{lemma}

\begin{proof}[Sketch of proof]
Under the given conditions, for any $(\lambda_{2}^r,c_{e})$, the monotonicity of $k_{2}(\cdot)$ implies the existence of a unique $\lambda_{e}\in[\mu_{0,2},\mu^{*}_{2,-})$ such that $k_{2}(\lambda_{e})=c_{e}$. In particular, we have $(\lambda_{e},c_{e})\in \gamma_{2}^e$ and $\lambda_{e}\leqslant \lambda_{2}^r$. Define
	\begin{equation*}
		\rho(s)=
	\begin{cases}
		\lambda_{e} s -H_{2,+}(\lambda_{e}) &\text{for~}s\in[c_{e},\infty),\\
		\max\left\{0,\mu^{*}_{2,-} s-H_{2,-}(\mu^{*}_{2})\right\}&\text{for~}s\in[0,c_{e}).
	\end{cases}
	\end{equation*}
Then   $c_{v}^r\leqslant c^{*}_{2,-}$.
\end{proof}

\begin{lemma}\label{lem:v-right-D}
	If either $\lambda_{2}^r\in(0,\mu_{0,2})$ and $c_{e}> c_{2,+}(\lambda_{2}^r)$, or $\lambda_{2}^r\in[\mu_{0,2},\mu^{*}_{2,-})$ and $c_{e}>k_{2}(\lambda_{2}^r)$,	then $c_{v}^r\leqslant c_{2,-}(p^{*}_2)$.
\end{lemma}

\begin{proof}[Sketch of proof] Define
	\begin{equation*}
		\rho(s)=
	\begin{cases}
		\lambda_{2}^rs-H_{2,+}(\lambda_{2}^r) & \text{for } s\in[c_{e},\infty),\\
	\max\left\{0,	p^{*}_{2}s-H_{2,-}(p^{*}_2)\right\} & \text{for } s\in[0,c_{e}).
	\end{cases}
	\end{equation*}
 We conclude that  $c_{v}^r\leqslant c_{2,-}(p^{*}_2)$ by a similar argument to \eqref{eq:right-C0}.
\end{proof}

\begin{proof}[\bf{Proof of Theorem \ref{thm:v-maxspeed-right}}]
	The result follows directly from Lemmas \ref{lem:right-division} and \ref{lem:v-right-A}-\ref{lem:v-right-D}.
\end{proof}

\subsection{Upper bound of the leftward spreading speed}\label{subsec:v-max-left}
In this subsection, we study the upper bound of the leftward spreading speed. The proofs of the following four lemmas are omitted as they are analogous to those in Subsection \ref{subsection-left-u}. These lemmas collectively yield Theorem \ref{thm:v-maxspeed-left}.

\begin{lemma}\label{lem:v-left-A}
	If $\lambda_{2}^l\in (0,\infty]$ and $c_{e}\geqslant -s_{2,-}^l$, then $c_{v}^l\geqslant -s_{2,-}^l$.
	\end{lemma}

	\begin{lemma}\label{lem:v-left-B}
		If either $\lambda_{2}^l\in(0,\mu^{*}_{2,+}]$ and $c_{e}<-c_{2,-}(\lambda_{2}^l)$, or $\lambda_{2}^l\in[\mu^{*}_{2,+},L'_{2}(\bar{c}_{2})]$ and $-g_{2}(\lambda_{2}^l)<c_{e}<-(c_{2,-}(\lambda_{2}^l)\vee H_{2}'(\lambda_{2}^l))$, then $ c_{v}^l\geqslant-c_{2,+}(\underline{p}_{2})$.
	\end{lemma}

	\begin{lemma}\label{lem:v-left-C}
		If $\lambda_{2}^l\in[\mu^{*}_{2,-},\infty]$ and $-(H'_{2}(\lambda_{2}^l)\land \bar{c}_{2})<c_{e}<-c^{*}_{2,-}$, then $c_{v}^l\geqslant-c_{2,+}(\bar{p}_{2})$.
	\end{lemma}

	\begin{lemma}\label{lem:v-left-D}
		If either $\lambda_{2}^l\in[\mu^{*}_{2,+},L'_{2}(\bar{c}_{2})]$ and $c_{e}\leqslant -g_{2}(\lambda_{2}^l)$, or $\lambda_{2}^l\in(L'_{2}(\bar{c}_{2}),\infty]$ and $c_{e}\leqslant -\bar{c}_{2}$, then $ c_{v}^l\geqslant -c_{2,+}^{*}$.
	\end{lemma}

	\section{Persistence of the prey}\label{sec:terrace-prey}
In this section, we prove Theorem \ref{thm:u-terrace}.

\begin{proof}[Proof of Theorem \ref{thm:u-terrace}] We only prove  \eqref{eq:u-terrace-1} and \eqref{eq:u-terrace-3}, as the other cases follow similarly.

First, we prove \eqref{eq:u-terrace-1} under the condition $c_{e}\geqslant s_{1,+}^r$.
By Theorems \ref{thm:u-speed-exp-right}-\ref{thm:u-speed-comp}, we have
	\begin{align*}
		c_{e}\geqslant c_{u}^r\geqslant s_{1,+}^r, \quad c_{u}^l=-s_{1,-}^l.
	\end{align*}
From assumption (FU) and Corollary \ref{coro:u-approach-max-capacity}, it holds that
\begin{align*}
	&\underset{t\rightarrow+\infty}\lim\, \underset{ (s_{2,-}^r+\eta)t \leqslant x \leqslant (c_{u}^r-\eta)t}\inf\, u(x,t)=\underset{\varepsilon\rightarrow0}\lim\, \underset{ s_{2,-}^r+\eta \leqslant y\leqslant c_{u}^r-\eta}\inf\, u\left(\frac{y}{\varepsilon},\frac{1}{\varepsilon}\right) \geqslant \alpha_{-}, \quad \forall \eta\in(0, \tfrac{c_{u}^r-s_{2,-}^r}2),\\
	&\underset{t\rightarrow+\infty}\lim\, \underset{ (-s_{1,-}^l+\eta)t \leqslant x \leqslant (-s_{2}^l-\eta)t}\inf\, u(x,t)=\underset{\varepsilon\rightarrow0}\lim\, \underset{-s_{1,-}^l+\eta \leqslant y\leqslant -s_{2}^l-\eta}\inf\, u\left(\frac{y}{\varepsilon},\frac{1}{\varepsilon}\right) \geqslant \alpha_{-}, \quad \forall \eta\in(0, \tfrac{s_{1,-}^l-s_{2,-}^l}{2}).
\end{align*}
Since $u\leqslant \alpha_{-}$ in $\mathbb{R}\times\mathbb{R}^{+}$, we conclude that \eqref{eq:u-terrace-1} holds.

Next, we prove  \eqref{eq:u-terrace-3} when  $-s_{1,-}^l<c_{e}<-s_{2,-}^l$. By Theorems \ref{thm:u-speed-exp-right}-\ref{thm:u-speed-comp}, we obtain
\begin{align*}
c_{u}^r=s_{1,+}^r, \quad c_{u}^l=-s_{1,-}^l.
\end{align*}
From Corollary \ref{coro:u-approach-max-capacity}, we have
\begin{align*}
	&\underset{t\rightarrow+\infty}\lim\, \underset{(s_{2,-}^r+\eta)t \leqslant x \leqslant (s_{1,+}^r-\eta)t}\inf\, u(x,t)=\underset{\varepsilon\rightarrow0}\lim\, \underset{ s_{2,-}^r+\eta \leqslant x \leqslant s_{1,+}^r-\eta}\inf\, u\left(\frac{y}{\varepsilon},\frac{1}{\varepsilon}\right) \geqslant \alpha_{+}, \quad \forall \eta\in(0, \tfrac{s_{1,+}^r-s_{2,-}^r}2),\\
	&\underset{t\rightarrow+\infty}\lim\, \underset{(c_{e}+\eta)t \leqslant x \leqslant (-s_{2,-}^l-\eta)t}\inf\, u(x,t)=\underset{\varepsilon\rightarrow0}\lim\, \underset{c_{e}+\eta\leqslant x \leqslant -s_{2,-}^l-\eta}\inf\,  u\left(\frac{y}{\varepsilon},\frac{1}{\varepsilon}\right) \geqslant \alpha_{+}, \quad \forall \eta\in (0,\tfrac{-s_{2,-}^l-c_{e}}{2}),\\
	&\underset{t\rightarrow+\infty}\lim\, \underset{ (-s_{1,-}^l+\eta)t \leqslant x \leqslant (c_{e}-\eta)t}\inf\, u(x,t)=\underset{\varepsilon\rightarrow0}\lim\, \underset{-s_{1,-}^l+\eta \leqslant y\leqslant c_{e}-\eta}\inf\, u\left(\frac{y}{\varepsilon},\frac{1}{\varepsilon}\right) \geqslant \alpha_{-}, \quad \forall \eta\in(0, \tfrac{c_{e}+s_{1,-}^l}2).
\end{align*}
Therefore, it remains to show that $u$ converges to $\alpha_{+}$ in the first two limits of \eqref{eq:u-terrace-3}. We only show the first one. Recall the auxiliary function $\varphi_{\delta}$ defined in \eqref{eq:var-p}, and choose $t_{0}>0$ sufficiently large so that $-c_{e}t>x_{\delta}$ for any $t>t_{0}$. Then for any $x\geqslant (s_{2,-}^r+\eta)t$ with  $t>t_{0}$, we have
\begin{align*}
	x-c_{e}t-x_{\delta}\geqslant (s_{2,-}^r+\eta-c_{e})t-x_{\delta}\geqslant s_{2,-}^r t.
\end{align*}
By \eqref{eq:limit-forced-wave} and the monotonicity of the forced wave, we obtain
\begin{align*}
	\underset{t\rightarrow+\infty}\lim\, \psi_{\delta}(x-c_{e}t-x_{\delta})\leqslant \underset{t\rightarrow+\infty}\lim\, \psi_{\delta}(s_{2,-}^r t)=\alpha_{+}+\delta.
\end{align*}
Since $\delta>0$ is arbitrary, \eqref{eq:u-terrace-3} follows immediately.
\end{proof}

\section*{Acknowledgments}

Research of W.-T. Li was partially supported by NSF of China (12531008; 12271226). Research of S. Ruan was partially supported by National Science Foundation (DMS-2424605). Research of W.-B. Xu was partially supported by  NSF of China (12201434;12371209) and R\&D  Program of Beijing Municipal Education Commission (KM202310028017).

\end{document}